\documentclass{article}
\usepackage{array}
\usepackage{tabularx}
\usepackage{amsmath,amssymb}
\usepackage{mathrsfs}
\usepackage[all]{xy}
\usepackage{amsthm}
\usepackage{comment}

\theoremstyle{definition}
\newtheorem{Def}{Definition}[section]
\newtheorem{Thm}[Def]{Theorem}
\newtheorem{Lem}[Def]{Lemma}
\newtheorem{Prop}[Def]{Proposition}

\newtheorem{Rem}[Def]{Remark}

\newtheorem*{Claim}{Claim}

\newcommand{\C}{\mathbb{C}}

\newcommand{\g}{\mathfrak{g}}
\newcommand{\ct}{\mathfrak{t}}

\newcommand{\n}{\mathfrak{n}}
\newcommand{\be}{\mathfrak{b}}
\newcommand{\OO}{\mathcal{O}}
\newcommand{\Otil}{\widetilde{\mathcal{O}}}
\newcommand{\Mtil}{\widetilde{M}}

\newcommand{\hd}{\mathop{\mathtt{hd}}\nolimits}

\newcommand{\std}{\Delta}
\newcommand{\pstd}{\bar{\Delta}}
\newcommand{\pcstd}{\bar{\nabla}}

\newcommand{\homo}{\mathop{\mathrm{hom}}\nolimits}
\newcommand{\ext}{\mathop{\mathrm{ext}}\nolimits}
\newcommand{\Ext}{\mathop{\mathrm{Ext}}\nolimits}
\newcommand{\End}{\mathop{\mathrm{End}}\nolimits}
\newcommand{\Hom}{\mathop{\mathrm{Hom}}\nolimits}
\newcommand{\Ker}{\mathop{\mathrm{Ker}}\nolimits}
\newcommand{\Ima}{\mathop{\mathrm{Im}}\nolimits}
\newcommand{\rad}{\mathop{\mathtt{rad}}\nolimits}

\newcommand{\m}{\mathfrak{m}}

\newcommand{\gl}{\mathfrak{gl}}
\newcommand{\SG}{\mathfrak{S}}

\newcommand{\CC}{\mathbf{C}}
\newcommand{\Cc}{\mathcal{C}}
\newcommand{\e}{\epsilon}
\newcommand{\Z}{\mathbb{Z}}
\newcommand{\KP}{\mathrm{KP}}

\newcommand{\supp}{\mathop{\mathrm{supp}}\nolimits}
\newcommand{\F}{\mathscr{F}}
\newcommand{\E}{\mathscr{E}}
\newcommand{\T}{\mathscr{T}}

\newcommand{\Fun}{\mathop{\mathrm{Fun}}\nolimits}
\newcommand{\gdim}{\mathop{\mathtt{gdim}}\nolimits}
\newcommand{\gmod}{\text{-}\mathtt{gmod}}
\newcommand{\mof}{\text{-}\mathrm{mod}}
\newcommand{\kk}{\Bbbk}
\newcommand{\eend}{\mathop{\mathrm{end}}\nolimits}
\newcommand{\rrad}{\mathop{\mathrm{rad}}\nolimits}
\newcommand{\hhd}{\mathop{\mathrm{hd}}\nolimits}

\numberwithin{equation}{section}

\title{Tilting modules
of affine quasi-hereditary algebras}
\author{Ryo Fujita\thanks{Department of 
Mathematics, Kyoto University, 
Oiwake Kita-Shirakawa Sakyo Kyoto 
606-8502 JAPAN, 
\texttt{E-mail:rfujita@math.kyoto-u.ac.jp}}}

\begin{document}

\maketitle

\begin{abstract}
	We discuss 
	tilting modules of affine quasi-hereditary algebras.
	We present an existence theorem of
	indecomposable tilting modules 
	when the algebra has a large center
	and use it to 
	deduce a criterion for an exact functor between 
	two affine highest weight categories to
	give an equivalence.
	As an application, we prove that
	the Arakawa-Suzuki functor 
	(Arakawa-Suzuki, 1998) gives a
	fully faithful embedding of  a block of
	the deformed BGG category of $\gl_{m}$
	into the module category of
	a suitable completion of degenerate affine Hecke algebra
	of $GL_{n}$.
\end{abstract}


\section{Introduction}  
	
	The notion of highest weight category
	and its ring theoretic counterpart,  
	quasi-hereditary algebra  
	introduced by Cline-Parshall-Scott \cite{CPS}
	enables us to study 
	representation theory 
	of algebras 
	of Lie theoretic origin
	in terms 
	of theory of
	Artin algebras.
	Ringel's influential work \cite{Ringel} 
	on tilting modules of quasi-hereditary algebras
	is one of the examples.
	In a highest weight category,
	there are two kinds of distinguished indecomposable
	modules called standard modules
	and costandard modules.
	In \cite{Ringel}, Ringel 
	presented tilting modules 
	which are characterized to be 
	filtered by both standard modules 
	and costandard modules.
	
	Recently,
	Kleshchev \cite{affine} defined the
	notion
	of affine highest weight category
	and affine quasi-hereditary algebra
	as a graded version of the definition 
	of Cline-Parshall-Scott \cite{CPS}.
	Examples 
	of affine highest weight categories
	include
	the graded module categories 
	of Khovanov-Lauda-Rouquier (KLR) algebras
	of finite Lie type (\cite{BKM}, \cite{Kato}),
	those of current Lie algebras etc. 
	There are two kinds of 
	counterparts to
	standard modules
	in an affine highest weight category
	called standard modules 
	(which are of infinite length in general)
	and 
	proper standard modules 
	(which are of finite length)
	respectively.
	The counterparts to costandard modules
	are called
	proper costandard modules (which are of finite length).
	
	In this paper, we discuss tilting modules 
	of affine quasi-hereditary algebras,
	which are defined to be filtered
	by both standard modules
	and proper costandard modules.
	Under some conditions on
	the center,
	we prove an existence 
	theorem
	of 
	indecomposable tilting modules
	and
	deduce a simple criterion
	for an exact functor between two 
	affine highest weight categories 
	to give an equivalence.
	
	For more detailed explanation, 
	let $H$ be an affine quasi-hereditary algebra
	over a field $\kk$.
	By definition, $H$ has a $\Z$-grading
	$H = \bigoplus_{n \in \Z}H_{n}$
	such that we have $\dim H_{n} < \infty$
	for any $n \in \Z$ and
	$H_{n} = 0$ for $n \ll 0$.
	Let
	$\{ \std(\pi) \mid \pi \in \Pi \}$
	be the set of standard modules of $H$,
	where $\Pi$ is a finite set 
	with a partial order $\leq$
	parameterizing simple modules of $H$.  
	For any finitely generated graded
	$H$-module $V$ with $\std$-filtration
	and $\pi \in \Pi$,
	we consider
	the number $(V:\std(\pi))$ 
	of appearance of some grading shift of
	$\std(\pi)$ as a subquotient,
	which is known to be finite.
	We consider the following property $(\spadesuit)$:
	
	\begin{itemize}
	\item[$(\spadesuit)$]
	There is a central subalgebra $Z \subset H_{\geq 0}$
	with $Z_{0} = \kk \cdot 1$ such that $H$		
	is finitely generated as a $Z$-module.
	 \end{itemize}
	 Note that this is equivalent to simply saying that
	 the algebra $H$ is a finitely generated module 
	 over its center. 
	 
	 Our main results are the followings: 
	 
	 \begin{Thm}[=Theorem \ref{Thm:tilt}] 
	 \label{Thm:main1}
	 Let $H$ be an affine quasi-hereditary algebra
	with property $(\spadesuit)$.
	Then for each $\pi \in \Pi$,
	there exists a unique indecomposable tilting module
	$T(\pi) $ such that
	$(T(\pi):\std(\pi))=1$
	and $(T(\pi):\std(\sigma))=0$ 
	for any $\sigma \not \leq \pi$.
	Moreover every indecomposable 
	tilting module is isomorphic to
	a grading shift of $T(\pi)$
	for some $\pi \in \Pi$.
	 \end{Thm}
	 
	 \begin{Thm}[=Theorem \ref{Thm:equiv}]
	 \label{Thm:main2}
	Suppose that an exact functor $\Phi$
	between the module categories over 
	affine quasi-hereditary algebras with
	property $(\spadesuit)$
	induces order-preserving bijections
	on
	standard modules
	and on proper costandard modules.
	Then $\Phi$ gives an equivalence
	of graded categories. 
	\end{Thm}
	
	 For example, 
	 the property $(\spadesuit)$ is satisfied by
	 KLR algebras.
	 Therefore Theorem \ref{Thm:main1}
	 says that 
	 for each affine quasi-hereditary structure
	 of KLR algebra of finite Lie type,
	 there exists a complete collection of 
	 indecomposable tilting modules 
	 (in the sense of Ringel).
	 
	 As an application of Theorem \ref{Thm:main2}, 
	we prove that
	the Arakawa-Suzuki functor \cite{AS} gives a
	fully faithful embedding of a block of
	the deformed BGG category of $\gl_{m}$
	into the module category of
	a suitable completion of degenerate affine Hecke algebra
	$H_{n}$
	of $GL_{n}$.
	
	The organization of this paper is as follows: 
	Section \ref{affhwc} is preliminary and 
	a brief review 
	of affine highest weight categories 
	following Kleshchev \cite{affine}.
	After a homological characterization
	of tilting modules
	(\ref{char}), 
	we prove the existence theorem
	of indecomposable tilting modules 
	(= Theorem \ref{Thm:main1}) in   
	Subsection \ref{tilt}.
	Our argument is similar to
	Donkin's exposition \cite{Donkin}.
	Then we prove
	Theorem \ref{Thm:main2}
	in Subsection \ref{equiv}.
	In Section \ref{Completion},
	we consider affine highest weight categories
	and affine quasi-hereditary algebras in
	the setting of topologically complete algebras.
	We also discuss the completion of 
	(graded) affine quasi-hereditary algebras 
	with property $(\spadesuit)$.
	In Section \ref{application}, we apply 
	the results of Section \ref{tilting} and \ref{Completion}
	to 
	the 
	Arakawa-Suzuki functor \cite{AS}, \cite{Suzuki}.
	We discuss an affine quasi-hereditary structure of a
	central completion of the
	degenerate affine Hecke algebra $H_{n}$ (\ref{dAHA})
	and the affine highest weight structure of the
	deformed BGG category of $\gl_{m}$ (\ref{defBGG}).
	Then we apply the complete version
	of Theorem \ref{Thm:main2} 
	to prove that the Arakawa-Suzuki functor
	is a fully faithful functor
	between two affine
	highest weight categories (\ref{AS}).   
	 
\section*{Acknowledgements}
	The author is very grateful to 
	Syu Kato
	for many discussions and encouragements.
	He also thanks Ryosuke Kodera 
	for helpful conversations and comments.

\section{Preliminary}
\label{affhwc}

\subsection{Overall notation}

	Fix a field $\kk$ throughout Section
	\ref{affhwc}, \ref{tilting} and 
	\ref{Completion}.
	In Section \ref{application}, we consider the case
	$\kk = \C$.
	For a $\kk$-algebra $H$, we write
	$H \text{-Mod}$ 
	to indicate the 
	$\kk$-linear category
	of all left $H$-modules
	and we write 
	$H \mof$ to indicate the 
	$\kk$-linear category
	of finitely generated left $H$-modules.
	The category $H \mof$ is abelian
	if $H$ is left Noetherian.
	We always assume that a $\kk$-linear 
	abelian category
	is Schurian i.e.\ 
	endomorphism ring of a simple object
	is isomorphic to $\kk$.
	The following general lemma will be used in the sequel.
	\begin{Lem}[cf.\  \cite{BKM} Lemma 1.1]
	\label{Lem:limvan}
	Let $\Cc$ be an abelian category and 
	$U, V \in \Cc$.
	We fix a non-negative integer $i$.
	Suppose that 
	 for a filtration
	$V = V_{0} \supset V_{1} \supset \cdots$
	in $\Cc$,
	we have a natural isomorphism
	$\displaystyle \varprojlim_{n}V/V_{n} \cong 
	V$ and 
	$\Ext_{\Cc}^{i}(U, V_{n}/V_{n+1}) = 0$
	for all $n \geq 0$.
	Then we have 
	$\Ext_{\Cc}^{i}(U, V)=0.$
	\end{Lem}

\subsection{Notation for graded categories}

	We always refer gradings as
	$\Z$-gradings.
	A graded category is
	a $\kk$-linear category $\CC$
	equipped with a
	self-equivalence $q$ called the grading shift functor
	and its quasi-inverse $q^{-1}$.
	For objects $V, V'$ in the graded category $\CC$,
	we write $V \simeq V'$ to indicate
	$V \cong q^{n}V'$ for some integer $n$.
	We set 
	$\homo_{\CC}(V, V')$ to be 
	the direct sum of spaces $
	\homo_{\CC}(V, V')_{n} :=
	\Hom_{\CC}(q^{n}V, V')$
	of graded homomorphisms of degree $n \in \Z$ and 
	define
	$\ext_{\CC}^{i}(V, V')$ similarly.
	A functor between graded categories 
	is said to be graded if it commutes with
	grading shift functors.
	
	For a graded $\kk$-algebra $H$,
	we denote by $H \gmod$ the category
	of all finitely generated graded left $H$-modules
	$V = \bigoplus_{\n \in \Z}V_{n}$
	whose morphisms 
	preserve the gradings.
	The category 
	$H \gmod$ is a graded category
	whose grading shift functor is defined as
	$(qV)_{n} = V_{n-1}.$
	We put $\homo_{H}(V,V')_{n}:=
	\Hom_{H \gmod}(q^{n}V,V')$. 
	In Section \ref{affhwc} and \ref{tilting},
	any modules (resp.\ ideals) 
	are assumed to be graded (resp. homogeneous).
	
 	A graded vector space $V = \bigoplus_{n \in \Z} V_{n}$
 	is said to be  Laurentian if we have
 	$V_{-n} = 0$ for $n \gg 0$ and
 	$\dim_{\kk} V_{n} < \infty$ for all $n \in \Z$.
 	For a Laurentian graded vector space
 	$V$, its graded dimension is defined as
 	$\gdim_{q}V := \sum_{n \in \Z} (\dim V_{n})q^{n}
 	\in \Z (\!(q)\!).$
 	A graded $\kk$-algebra is said to be Laurentian
 	if its underlying graded vector space is Laurentian.
 	For a Laurentian graded algebra $H$, we write 
	$\hd V$ (resp.\ $\rad V$)
	for the head (resp.\ radical) of an object $V$
	in $H \gmod$.

\subsection{Affine highest weight category}

	In this subsection,
	we review the definition and some properties of
	a (graded) affine highest weight category
	following Kleshchev \cite{affine}.

	Let $\CC$ be
	a graded category with a set
	$\{L(\pi) \mid \pi \in \Pi \}$
	of simple objects,
	which is complete and irredundant up to
	isomorphisms and grading shifts.
	\begin{Def} \label{Def:NLC}
	A graded abelian category $\CC$ is called a
	Noetherian Laurentian category if
	the following three conditions hold:
		\begin{enumerate}
		\item  \label{Def:NLC:filt}
		Every object $V$ in $\CC$ is Noetherian
		and has
		a filtration
		 $V \supset V_1 \supset V_{2} \supset \cdots$
		which is separated 
		(i.e.\ $\bigcap_{n=1}^{\infty}V_{n} = 0$)
		such that each quotient $V/V_n$ has finite length;
		\item Each simple object
		$L(\pi)$ has a projective cover
		$P(\pi) \twoheadrightarrow L(\pi)$;
		\item For all $\pi, \sigma \in \Pi$,
		the graded vector space
		$\homo_{\CC}(P(\pi), P(\sigma))$
		is Laurentian.  
		\end{enumerate}
	\end{Def}
	For example,
	the category $H \gmod$ for
	a left Noetherian Laurentian $\kk$-algebra $H$
	is a Noetherian Laurentian category.
	
	The following lemma easily follows from  
	\cite{affine} Lemma 3.3 (iii).
	\begin{Lem}
	\label{Lem:NLC}
		Let $U, V$ be objects of a Noetherian
		Laurentian category $\CC$ and 
		$V \supset V_{1} \supset V_{2} \supset \cdots$
		be a filtration as 
		in Definition \ref{Def:NLC} (\ref{Def:NLC:filt}).
		Then:
		\begin{enumerate}
		\item \label{Lem:NLC:fin}
		there exists an integer $N$ such that
		$\Hom_{\CC}(U, V_{n})=
		\Ext^{1}_{\CC}(U, V_{n})
		=0$ for all $n \geq N$;
		\item \label{Lem:NLC:lim}
		we have a natural isomorphism
		$\displaystyle
		\varprojlim_{n}\Hom_{\CC}
		(U, V/V_{n})
		\cong
		\Hom_{\CC}(U, V).
		$\qed
		\end{enumerate}
	\end{Lem}

	We assume
	the set $\Pi$ to be equipped with a partial order $\leq$.
	For each $\pi \in \Pi$, we define
	the standard object $\std(\pi)$ 
	and the proper standard object $\pstd(\pi)$ as:
	\begin{align*}
	  \std(\pi) &:= P(\pi)/ (
	  \sum_{\sigma \not \leq 
	  \pi, f \in \homo_{\CC}(P(\sigma), P(\pi))}
	  \Ima f ),\\
	\pstd(\pi) &:= P(\pi) / (
	  \sum_{\sigma \not < \pi, f \in \homo_{\CC}(P(\sigma), 
	  \rad P(\pi))}
	  \Ima f ).
	\end{align*}
	We say that an object $V \in \CC$
	has a $\std$-filtration if $V$ has a separated filtration
	$V=V_{0} \supset V_{1} \supset \cdots$
	whose subquotients are
	$\simeq \std(\pi)$ for some $\pi \in \Pi$.

	\begin{Def} \label{Def:hwc}
	A Noetherian Laurentian category $\CC$
	is called an affine highest weight category
	if the following three conditions hold:
	\begin{enumerate}
	\item \label{Def:hwc:str}
	For each $\pi \in \Pi$,
	the kernel of the natural quotient map
	$P(\pi) \twoheadrightarrow \std(\pi)$
	has a $\std$-filtration whose subquotients are 
	$ \simeq \std(\sigma)$ with $\sigma > \pi$;

	\item \label{Def:hwc:end}
	For each $\pi \in \Pi$, the algebra
	$B_{\pi} := \eend_{\CC}(\std(\pi))$
	is isomorphic to a graded polynomial ring
	$\kk [ z_{1}, \ldots , z_{n_{\pi}}]$
	for some $n_{\pi} \in \Z_{\geq 0}$
	with $\deg z_{i} \in \Z_{>0};$
	
	\item \label{Def:hwc:free}
	For any $\pi, \sigma \in \Pi$,
	the $B_{\sigma}$-module
	$\homo_{\CC}(P(\pi), \std(\sigma))$
	is graded free of finite rank. 
	\end{enumerate}
	\end{Def}
	
	\begin{Rem} \label{Rem:gldim}
		Kleshchev's definition of affine highest weight category
		in \cite{affine} does not require that 
		each $B_{\pi}$ is a polynomial ring.
		However, we require it
		in order to
		guarantee that the global dimension 
		of $H$ is finite by \cite{affine}
		Corollary 5.25, 
		which is needed in Subsection \ref{equiv}.     
	\end{Rem}
	
	Any affine highest weight category $\CC$
	with finite index poset $\Pi$ is
	equivalent to a graded module category
	$H \gmod$ over
	a left Noetherian Laurentian algebra $H$.
	Such an algebra $H$ is
	characterized 
	as an affine quasi-hereditary algebra,
	without referring to its module category.
	See \cite{affine} Section 6 for the definition.
	More precisely,
	we have the following.

	\begin{Thm}[\cite{affine} Theorem 6.7] 
	\label{Thm:affineqh}
	For a Noetherian Laurentian algebra $H$,
	the category $H \gmod$ 
	is an affine highest weight category
	if and only if
	the algebra $H$ is an affine quasi-hereditary algebra.
	\end{Thm}
	
	From now on, we assume that the index poset 
	$\Pi$ is finite.
	Thanks to Theorem \ref{Thm:affineqh},
	we can use the expressions 
	``$H$ is an affine quasi-hereditary algebra''
	and ``$H \gmod$ is an affine highest weight category''
	interchangeably. 
	
	For each $\pi \in \Pi$,
	let $I(\pi)$ be an injective hull of $L(\pi)$
	in the category of all graded $H$-modules.
	Let $A(\pi)$ be
	the largest
	submodule of $I(\pi)/L(\pi)$ among all of whose
	composition factors are of the form
	$\simeq L(\sigma)$ for $\sigma < \pi$.
	We define the proper costandard module
	$\pcstd(\pi) \subset I(\pi)$ to be the preimage of $A(\pi)$
	with respect to the quotient map 	$I(\pi) \to I(\pi)/L(\pi)$.

	\begin{Thm}[\cite{affine} Theorem 7.6]
	\label{Thm:BHBGG}
	For each $\pi \in \Pi $, 
	the module $\pcstd(\pi)$ 
	has finite length and
	characterized by the property:
	$$ \displaystyle
	\ext_{H}^{i}(\std(\sigma),\pcstd(\pi)) =
					\begin{cases}
					\kk
					& i=0,\sigma = \pi ;\\
					0 & \text{otherwise}.
					\end{cases}
	$$
	\end{Thm}
	
	Thanks to Lemma \ref{Lem:NLC}
	(\ref{Lem:NLC:fin}) 
	and Theorem \ref{Thm:BHBGG},
	we know that the length of any $\std$-filtration 
	of $V \in  H \gmod$
	is finite and the multiplicity of 
	grading shifts of $\std(\pi)$ is equal to
	$(V:\std(\pi)):= \gdim_{q} 
	\homo_{H}(V, \pcstd(\pi)) |_{q=1} .$ 
	
	We say that an object $V \in \CC$
	has a $\pcstd$-filtration if $V$ has a separated filtration
	$V=V_{0} \supset V_{1} \supset \cdots$ 
	whose subquotients are
	$\simeq \pcstd(\pi)$ for some $\pi \in \Pi$.
	
	The following lemma is frequently used in the sequel.
	
	\begin{Lem}[\cite{affine} Lemma 5.10 and 7.7]
	 \label{Lem:ext}
	 Let $\sigma, \pi \in \Pi$.
	 \begin{enumerate}
	 \item \label{Lem:ext:std}
	 If $\sigma \not < \pi$, then
		$\ext^{1}_{H}(\std(\sigma), \std(\pi))=0;$
	\item \label{Lem:ext:pcstd}	
	If $\sigma < \pi$, then
	$\ext_{H}^{1}(\pcstd(\sigma), \pcstd(\pi)) = 0.$
	\end{enumerate}
	\end{Lem}
	
\section{Tilting modules}
\label{tilting}

	We keep the notation in the previous section.

	Let $H$ be an affine quasi-hereditary algebra
	with the index poset $\Pi$.
	We write $\F(\std)$ (resp.\ $\F(\pcstd)$)
	to indicate the full subcategory
	of $H \gmod$ consisting of
	all $\std$-filtered (resp.\ $\pcstd$-filtered) modules.

	\begin{Def}
		A module $V \in H \gmod$ is called a
		tilting module
		if it is both $\std$-filtered and 
		$\pcstd$-filtered, i.e. it belongs to
		$\T := \F(\std) \cap \F(\pcstd).$
	\end{Def}
	
\subsection{Homological characterization}
\label{char}	
	\begin{Thm} \label{Thm:char}
	Let $H$ be an affine quasi-hereditary algebra
	and $V \in H \gmod.$
	\begin{enumerate}
	
	\item \label{Thm:char:std}
	$V$ belongs to $\F(\std)$ if and only if
	$\ext^{1}_{H}(V, \pcstd(\pi)) = 0$ for all $\pi \in \Pi$;
	
	\item \label{Thm:char:pcstd}
	$V$ belongs to $\F(\pcstd)$
	if and only if
	$\ext^{1}_{H}(\std(\pi), V) = 0$ for all $\pi \in \Pi$.
	
	\end{enumerate}
	\end{Thm}

	\begin{proof}
	(\ref{Thm:char:std}) is \cite{affine} Lemma 7.8.
	A proof of (\ref{Thm:char:pcstd})
	will be given after Lemma \ref{Lem:saturated}.
	\end{proof}
	
	\begin{Rem} \label{Rem:char}
	For the case where $V$ is of finite length, 
	Theorem \ref{Thm:char} (\ref{Thm:char:pcstd}) 
	has been proved by Kleshchev 
	\cite{affine} Lemma 7.9.  
	\end{Rem}
	
	For a module $V \in H \gmod$, we define its support set
	$\supp(V) \subset \Pi$ by
	$
	 \supp(V) := 
	 \{\pi \in \Pi \mid \homo_{H}(P(\pi), V) \neq 0 \}.
	$
	 We say that a subset $A$ of $\Pi$
	 is saturated
	 if it has the property
	 that $\pi \in A$ whenever
	 $\pi \leq \sigma$ and $\sigma \in A$.

	\begin{Lem} \label{Lem:saturated}
	Let $V \in H \gmod$
	and  fix a separated filtration
	$V = V_{0}\supset V_{1}  \supset \cdots$
	whose subquotients $V_{r}/V_{r+1}$ are simple
	for all $r \geq 0$.
	Assume that
	$\ext_{H}^{1}(\std(\sigma), V) = 0$ for all $\sigma \in \Pi.$
	Then for each saturated subset $A \subset \Pi$,
	there exists a filtration
	$V = W(A)_{0} \supset W(A)_{1} \supset \cdots$
	which satisfies the following two conditions:
	\begin{itemize}
	\item[$(\alpha)$] For each $r \geq 0$, 
	the subquotient $W(A)_{r}/W(A)_{r+1}$
	has a $\pcstd$-filtration of finite length
	whose subquotients are of the form 
	$\simeq \pcstd(\sigma)$
	for some $\sigma \in \Pi \setminus A$;
	\item[$(\beta)$]
	For each $n \geq 0$, we have
	$\supp(W(A)_{N}/(V_{n}\cap W(A)_{N})) \subset A$
	for $N \gg 0$.
	\end{itemize}
	\end{Lem}
	\begin{proof}
		We proceed by induction on the size of $\Pi 
		\setminus A$.
		The case where $A = \Pi$ is trivial
		by setting $W(A)_{i} := V$ for every $i \geq 0$.

		Let $\pi \in \Pi \setminus A$ be a minimal element
		and set $A' := A \cup \{ \pi \}$.
		Then $A'$ is a saturated subset 
		and $|A'| = |A|+1$.
		By induction hypothesis,
		there exists a filtration
		$V = W(A')_{0} \supset W(A')_{1} \supset \cdots$
		satisfying the conditions $(\alpha)$
		and $(\beta)$ with respect to $A'$.
		
		Since for each $n$ the quotient
		$V/V_{n}$ is finite length,
		we find a sequence of integers
		$0<N_{1} < N_{2} < \cdots$ such that
		for each $n$ we have natural isomorphisms
		$W(A')_{N_{n}}/V_{n}\cap W(A')_{N_{n}} \cong 
		W(A')_{r}/V_{n}\cap W(A')_{r}$
		for all $r \geq N_{n}$.
		Note that 
		we have a natural surjective map
		$$\frac{W(A')_{N_{n+1}}}{V_{n+1} \cap W(A')_{N_{n+1}}}
	\twoheadrightarrow
	\frac{W(A')_{N_{n+1}}}{V_{n} \cap W(A')_{N_{n+1}}}
	\cong
	\frac{W(A')_{N_{n}}}{V_{n} \cap W(A')_{N_{n}}},
	$$
	for each $n \geq 0.$
		The projective limit
		$$J(A')
		:= \varprojlim_{n} \frac{W(A')_{N_{n}}}
		{V_{n} \cap W(A')_{N_{n}}}
		$$
		is contained in 
		$
		W(A')_{m}
		$
		for all $m$.
		By construction
		and the condition $(\beta)$, we have
		$\supp J(A') \subset A'.$
		For each $\sigma \in A'$,
		we apply the functor $\homo_{H}(\std(\sigma), -)$
		to the short exact sequence 
		$0 \to J(A') \to V \to V/J(A') \to 0$
		to get:
		$$
		\homo_{H}(\std(\sigma), V/J(A')) 
		\to 
		\ext^{1}_{H}(\std(\sigma), J(A'))
		\to 
		\ext^{1}_{H}(\std(\sigma), V) = 0. 
		$$
		We observe that
		$\homo_{H}(\std(\sigma), V/J(A')) =		
		0$,
		which follows from the fact
		$V/ J(A')$ has a $\pcstd$-filtration 
		whose subquotients are 
		$\simeq \pcstd(\rho)$ for some $\rho \in \Pi 
		\setminus A'$
		and Theorem \ref{Thm:BHBGG}.
		Thus we have $\ext^{1}_{H}(\std(\sigma), J(A')) = 0$
		for all $\sigma \in A'.$

		If we have
		$\supp(W(A')_{N_{n}}/(V_{n}\cap W(A')_{N_{n}})) 
		\subset A$
		for all $n \geq 0$, then
	 	put $W(A)_{n} := W(A')_{n}$ for all $n \geq 0$ and
	 	we are done.
	 	If not,
	 	we prove 
	 	the following claim:
	 	\begin{Claim}
	 	There is a strictly increasing sequence
		of integers
		$0 \leq n_1 < n_2 < \cdots$, and
		for each $i \in \Z_{\geq 0}$
		a filtration
		$V = W_{i}(A)_{0} \supset W_{i}(A)_{1} \supset \cdots$
		such that:
		\begin{itemize}
		\item[$(\alpha)_{i}$]
		For each $r \geq 0$, the subquotient
		$W_{i}(A)_{r}/W_{i}(A)_{r+1}$ has a $\pcstd$-filtration
		of finite length whose
		subquotients are of the form $\simeq \pcstd(\sigma) $
		for some $\sigma \in \Pi \setminus A$;
		\item[$(\beta)_{i}$]
		For each $0 \leq r \leq n_i$, we have
		$\supp(W_{i}(A)_{N}/(V_{r}\cap W_{i}(A)_{N}))
		\subset A$
		for sufficiently large $N$;
		\item[$(\gamma)_{i}$]
		$W_{i}(A)_r = W_{i-1}(A)_{r}$ for $ 0 \leq r < n_{i}. $
		\end{itemize}
		\end{Claim}
		\begin{proof}
		Let $n_1$ be the smallest integer
		satisfying
		$\pi \in \supp(W(A')_{N_{n_1}}/
		(V_{n_1}\cap W(A')_{N_{n_1}})) $.
		We put $W :=W(A')_{N_{n_1}}$.
		Since $\pi \in \supp(W/ (V_{n_1} \cap W))$
		and $\pi \not \in \supp(W/(V_{n_{1}-1}\cap W))$,
		we see that 
		there is a unique embedding
		$q^{m}L(\pi) \hookrightarrow W/(V_{n_{1}} \cap W)$
		for some $m \in \Z$.
		We lift the natural embedding
		$q^{m}L(\pi) \hookrightarrow
		q^{m}I(\pi)$
		to a morphism
		$\phi : W/(V_{n_1} \cap W) \to q^{m}I(\pi)$
		so that we have
		$X := \Ima(\phi) \subset q^{m}\pcstd(\pi) 
		\subset q^{m}I(\pi)$.
		We set
		$Y := \Ker(J(A') \twoheadrightarrow
		W/(V_{n_1} \cap W)
		 \overset{\phi}{\twoheadrightarrow} X)$
		and get an exact sequence:
		\begin{equation}
		0 \to Y \to J(A') \to X \to 0.
		\label{eq:1}
		\end{equation}
		For each $\sigma \in A'$,
		we
		apply the functor
		$\homo_{H}(\std(\sigma), -)$
		to (\ref{eq:1})
		to obtain an exact sequence
		$$
		\homo_{H}(\std(\sigma), J(A')) \to
		\homo_{H}(\std(\sigma), X) \to
		\ext^{1}_{H}(\std(\sigma), Y) \to 
		\ext^{1}_{H}(\std(\sigma), J(A')) = 0.
		$$
		Since $X \subset q^{m} \pcstd(\pi)$,
		Theorem \ref{Thm:BHBGG} implies that
		$\homo_{H}(\std(\sigma), X) = 0$ for 
		$\sigma \neq \pi$.
		Therefore we have
		$\ext^{1}_{H}(\std(\sigma), Y) = 0$ 
		for $\sigma \in A = A' \setminus \{ \pi \}$.
		When $\sigma = \pi$,
		we have
		$$
		\homo_{H}(\std(\pi), J(A')) \stackrel{a}{\to}
		\homo_{H}(\std(\pi), X) \to
		\ext^{1}_{H}(\std(\pi), Y) \to 0.
		$$
		Note that  
		the full subcategory 
		$H \gmod_{A'} \subset H \gmod$
		consisting of  
		modules $V$
		with $\supp(V) \subset A'$
		is also an affine highest weight category
		with the set of 
		standard objects $\{\std(\sigma) \mid 
		\sigma \in A' \}$
		(cf.\ \cite{affine} Proposition 5.16). 
		The map $a$ is surjective because 
		$\pi$ is maximal in $A'$ and 
		so $\std(\pi)$ is projective in 
		$H \gmod_{A'}.$
		Thus we have   
		$\ext_{H}^{1}(\std(\pi), Y)=0.$
		
		We have seen that
		$\ext^{1}_{H}(\std(\sigma), J(A')) =
		\ext^{1}_{H}(\std(\sigma), Y) = 0$
		for all $\sigma \in A'$.
		We can apply a standard argument
		as in \cite{Hum} Section 6.12 to find that 
		$\ext^{i}_{H}(\std(\sigma), J(A')) =
		\ext^{i}_{H}(\std(\sigma), Y) = 0$
		for all $\sigma \in A'$ and $i \geq 1$.
		By considering the long exact sequence induced from
		the short exact sequence (\ref{eq:1}),
		we conclude that
		$\ext_{H}^{i}(\std(\sigma), X) = 0$
		 for all $\sigma \in A'$
		and $i \geq 1$.
		Since $X$ is of finite length,
		we can apply Theorem
		\ref{Thm:char} (\ref{Thm:char:pcstd}) 
		to $X$ in $H \gmod_{A'}$ (see Remark \ref{Rem:char})
		to 
		find that $X = q^{m}\pcstd(\pi)$.

		We set $W_{0}(A)_{r}:=W(A')_{r}$ for all $r\geq 0 $ and
		define
		$$W_{1}(A)_{r} :=
		\begin{cases}
					W_{0}(A)_{r}
					& 0 \leq r < n_{1} ;\\
					W_{0}(A)_{r} \cap Z & n_1 \leq r ,
					\end{cases}
		$$
		where 
		$Z := \Ker(W \twoheadrightarrow W/W\cap V_{n_{1}} 	
		\twoheadrightarrow X=q^{m}\pcstd(\pi) ).$
		Then we have a filtration
		$V = W_{1}(A)_{0} \supset W_{1}(A)_{1} \supset \cdots$
		satisfying the conditions 
		$(\alpha)_{1}, (\beta)_{1}$ and $(\gamma)_{1}$.
		By iterating the same argument, 
		we prove the claim inductively.
		\end{proof}

		 To complete the proof of Lemma \ref{Lem:saturated},
		 we set
		 $\displaystyle W(A)_{r} := \lim_{i \to \infty}W_{i}(A)_{r}$
		 for each $r \in \Z_{\geq 0},$
		 which is well-defined by the condition $(\gamma)_{i}$.
		 Then the filtration
		 $V=W(A)_{0} \supset W(A)_{1} \supset \cdots$
		 satisfies the required condition
		 $(\alpha)$ and $(\beta)$.
 	\end{proof}
 	
 	\begin{proof}
	[Proof of Theorem \ref{Thm:char} (\ref{Thm:char:pcstd}) ]
		Assume that $V$ has a $\pcstd$-filtration,
		say $V = V_{0} \supset V_{1} \supset \cdots.$
		We have
		$\ext_{H}^{1}(\std(\pi), V_{n}/ V_{n+1})=0$ 
		for all $\pi \in \Pi$
		and $n \geq 0$ by Theorem \ref{Thm:BHBGG}.
		Then by
		Lemma \ref{Lem:limvan},
		we have $\ext_{H}^{1}(\std(\pi), V)=0$ 
		for all $\pi \in \Pi$.
		The converse implication is a special case of
		Lemma \ref{Lem:saturated} where
		$A=\varnothing$.
	\end{proof}

\subsection{Existence theorem}
\label{tilt}	
	We consider the following property for 
	a left Noetherian Laurentian algebra
	$H$.
	\begin{itemize}
		\item[$(\spadesuit)$]
		There is a central 
		$\Z_{\geq 0}$-graded
		subalgebra $Z \subset H_{\geq 0}$
		with $Z_{0} = \kk \cdot 1$ such that $H$
		is finitely generated as a $Z$-module.
	 \end{itemize}
	Note that $Z$ is a graded local ring with maximal ideal 
	$\m:=Z_{>0}$. For a module $M \in Z \gmod$, we have
	$\hd M = M/ \m M$, which is finite-dimensional over $\kk$. 
	If $H$ is a Noetherian Laurentian algebra
	with property $(\spadesuit)$, then
	$\ext_{H}^{i}(U,V) \in Z \gmod$ 
	for any $U,V \in H \gmod$ and $i \geq 0$.
	
	For a subset $A \subset \Pi$, we define 
	its closure by
	$\overline{A} := \{
	\pi \in \Pi \mid
	\pi \leq \sigma, \, \exists
	\sigma \in A
	\}.$

	\begin{Lem} \label{Lem:tilt}
	Let $H$ be an affine quasi-hereditary algebra
	with property $(\spadesuit)$.
		For each $\pi \in \Pi$, let 
	$\F_{\pi}(\std)$ (resp.\ $\T_{\pi}$) be the full subcategory 
	of $\F(\std)$ (resp.\ $\T$) 
	consisting of modules $V$
	satisfying that $(V:\std(\pi))=1$ and 
	$(V:\std(\sigma))=0$ for any $\sigma \not \leq \pi$.
	Then for each $V \in \F_{\pi}(\std)$,
	there is an embedding $V \hookrightarrow T$
	into some $T \in \T_{\pi}$. 
	 \end{Lem}
	 \begin{proof}
	 For a module $V \in H \gmod$,
	 we define a subset $\mathfrak{D}(V) \subset \Pi$ by
	 $
	 \mathfrak{D}(V) := \{ \sigma \in \Pi \mid
	 \ext_{H}^{1}(\std(\sigma), V) \neq 0
	 \}.
	 $
	 Assume to be the contrary 
	 to deduce a contradiction. 
	 Let $\E_{\pi}(\std) \subset \F_{\pi}(\std)$ be the full 
	 subcategory of modules $V \in \F_{\pi}(\std)$
	 which have no embeddings into any tilting modules
	 belonging to $\T_{\pi}$.
	 Let $A \subset \Pi$ be a subset
	 with the smallest size $|A|$ 
	 among those 
	 of the form 
	 $A = \overline{\mathfrak{D}(V)} $
	 for some $
	 V \in \E_{\pi}(\std) $.
	 Note that we have 
	 $\sigma < \pi$ for any
	 $\sigma \in A$ because of 
	 Lemma \ref{Lem:ext} (\ref{Lem:ext:std}).
	 If $A = \varnothing$,
	 there is a module $V \in \E_{\pi}(\std)$ with
	 the property
	 $\ext_{H}^{1}(\std(\sigma), V) = 0$ for
	 all $\sigma \in \Pi$.
	 Then $V \in \F(\pcstd)$ by Theorem
	 \ref{Thm:char} (\ref{Thm:char:pcstd}) and thus
	 $V$ already belongs to $\T_{\pi}$,
	 which contradicts to our assumption.
	 For the case $A \neq \varnothing$,
	 we fix a maximal element $\rho \in A$ and pick
	 a module $V \in \E_{\pi}(\std)$ so that
	 $\overline{\mathfrak{D}(V)} = A$
	 and  
	 $\dim_{\kk} (\hd \ext_{H}^{1}(\std(\rho), V))$
	 is as small as possible.
	 Let $\theta \in \Ext_{H}^{1}(q^{m}\std(\rho), V)$ 
	 be an element
	 such that
	 $\theta \not \equiv 0 
	 \bmod \m \ext_{H}^{1}(\std(\rho), V).$
	 Consider the extension
	 $
	 0 \to  V \to V' \to q^{m}\std(\rho) \to 0
	 $
	  corresponding to $\theta$.
	  Then we have $V' \in \F_{\pi}(\std)$ and for any
	  $\sigma \in \Pi$ we have an exact sequence
	  in $Z \gmod$:
	  $$
	  \homo_{H}(\std(\sigma), q^{m}\std(\rho))
	  \to \ext_{H}^{1}(\std(\sigma), V)
	  \to \ext_{H}^{1}(\std(\sigma), V')
	  \to \ext_{H}^{1}(\std(\sigma), q^{m}\std(\rho)).
	  $$
	  
	  First, we claim that $\overline{\mathfrak{D}(V')} \subset A$.
	  When $\sigma \not \in A$,
	  we have
	  $\ext_{H}^{1}(\std(\sigma), V)= 0$ and
	  we have
	  $\ext_{H}^{1}(\std(\sigma), \std(\rho))=0$
	  since $\sigma \not < \rho $ and 
	  Lemma \ref{Lem:ext} (\ref{Lem:ext:std}).
	  Therefore we conclude that 
	  $\ext_{H}^{1}(\std(\sigma), V')=0$
	  for any $\sigma \not \in A$,
	  from which the claim follows. 
	  
	  Next, we consider the case $\sigma = \rho$.
	  We have $\ext_{H}^{1}(\std(\rho), q^{m}\std(\rho))=0$
	  by Lemma \ref{Lem:ext} (\ref{Lem:ext:std}).
	  By taking $\m$-coinvariants,
	  we get an exact sequence:
	  $$
	  q^{-m} \hd (B_{\rho}) \to
	  \hd \ext_{H}^{1}(\std(\rho), V)
	  \to
	\hd \ext_{H}^{1}(\std(\rho), V')
	\to 0,
	  $$
	  where the most left arrow is non-zero because of
	  our choice of the extension $\theta$.
	  Therefore, we have
	  $\dim_{\kk} (\hd \ext_{H}^{1}(\std(\rho), V'))
	  <
	  \dim_{\kk} (\hd \ext_{H}^{1}(\std(\rho), V)).$
	  By the minimality assumption,
	  we have $V' \not \in \E_{\pi}(\std).$
	  Thus there exists an embedding 
	  $V' \to T'$ into some $T' \in \T_{\pi}$.
	  Since $V$ is a submodule of $V'$, this contradicts to
	  our assumption.
	 \end{proof}

	 \begin{Thm} \label{Thm:tilt}
	 Let $H$ be an affine quasi-hereditary algebra
	with property $(\spadesuit)$.
	\begin{enumerate}
		\item \label{Thm:tilt:indec}
		For each $\pi \in \Pi$,
		there exists an indecomposable tilting module
		$T(\pi) \in H \gmod$ such that
		$(T(\pi):\std(\pi))=1$
		and $(T(\pi):\std(\sigma))=0$ 
		for any $\sigma \not \leq \pi$.
		Such $T(\pi)$ is unique up to isomorphisms
		and grading shifts.
		\item \label{Thm:tilt:local}
		For each $\pi \in \Pi$, the algebra
		$\eend_{H}(T(\pi))$ is a local ring with quotient 
		$\cong \kk$;
		\item \label{Thm:tilt:smd}
		Every tilting module is a finite direct sum 
		of the modules $q^{n}T(\pi)$
		for some $\pi \in \Pi$ and $n \in \Z$.
	 \end{enumerate}
	 \end{Thm}
	 \begin{proof}
	 We apply Lemma \ref{Lem:tilt} to $\std(\pi)$ to find
	 an embedding $\std(\pi)
	 \hookrightarrow T$ for some 
	 $T \in \T$ which satisfies
	 $(T:\std(\pi))=1$
	and $(T:\std(\sigma))=0$ for any $\sigma \not \leq \pi$.
	 By Theorem \ref{Thm:char},
	 the indecomposable direct summand $T(\pi)$ of $T$
	 with $(T(\pi):\std(\pi))=1$
	 is also a tilting module and 
	 satisfies the required property
	 in (\ref{Thm:tilt:indec}).
	 We prove the uniqueness of $T(\pi)$ later.
	 
	 Let us prove the assertion (\ref{Thm:tilt:local}).
	 By Lemma \ref{Lem:ext} (\ref{Lem:ext:std}), 
	 we have an embedding
	 $\std(\pi) \hookrightarrow q^{m}T(\pi)$
	 for some $m \in \Z$ such that
	 $T(\pi) / \std(\pi)$ also has a $\std$-filtration.
	 By grading shift, we may normalize so that $m=0$.
	 Thus
	 we have a short exact sequence:
	 $0 \to \std(\pi) \to T(\pi) \to Q(\pi) \to 0,$
	 where $Q(\pi)$ has a
	 $\std$-filtration
	 whose subquotients are of the form $\simeq \std(\sigma)$
	 for some $\sigma < \pi$.
	 We apply the functor $\homo_{H}(-, \pcstd(\pi))$ to get:
	 $$
	 \homo_{H}(Q(\pi), \pcstd(\pi))
	 \to \homo_{H}(T(\pi), \pcstd(\pi)) 
	 \to \homo_{H}(\std(\pi), \pcstd(\pi)) \to
	 \ext_{H}^{1}(Q(\pi), \pcstd(\pi)).
	 $$
	 By Theorem \ref{Thm:BHBGG}, the first and the fourth terms
	 are zero.
	 Therefore we have
	 $\homo_{H}(T(\pi), \pcstd(\pi))
	 \cong \homo_{H}(\std(\pi), \pcstd(\pi)) \cong \kk$.
	 This 
	 and Lemma \ref{Lem:ext} (\ref{Lem:ext:pcstd})
	 show that there exists a unique submodule 
	 $U(\pi) \subset T(\pi)$
	 such that $T(\pi) / U(\pi) \cong \pcstd(\pi).$
	 Then, by the graded version of 
	 Fitting's lemma, we can see that
	 $I_{\pi} := \{ f \in \eend_{H}(T(\pi))
	 \mid f(T(\pi)) \subset U(\pi) \} $
	 is the Jacobson radical of $\eend_{H}(T(\pi)) $
	 with $\eend_{H}(T(\pi)) /I_{\pi} \cong \kk.$
	 In particular, $\eend_{H}(T(\pi))$ is local.

	 To prove the assertion (\ref{Thm:tilt:smd})
	 and the uniqueness of $T(\pi)$'s,
	 it is enough to show that,
	 for each tilting module $T$, there is a
	 direct sum decomposition
	 $T \cong T' \oplus q^{m}T(\pi)$
	 for some $\pi \in \Pi, \, m \in \Z$
	 and $T' \in \T$.
	 Let $\pi $ be a maximal element in $\supp(T).$
	 Then Lemma \ref{Lem:ext} (\ref{Lem:ext:pcstd}) shows that 
	 there is an exact sequence:
	 $
	 0 \to U
	 \to
	 T \to q^{m} \pcstd(\pi) \to 0$
	 with $U \in \F(\pcstd)$.
	 Then we lift the
	 surjection $q^{m}T(\pi) \to q^{m}\pcstd(\pi)$
	 to a map $\phi : q^{m}T(\pi) \to T$.
	 By a similar argument,
	 we find a map $\psi : T \to q^{m}T(\pi)$
	 of the converse direction
	 which satisfies that
	 $\psi \circ \phi \equiv \mathrm{id} \bmod I_{\pi}$.
	 Then from the assertion 
	 (\ref{Thm:tilt:local})
	 proved in the previous  paragraph,
	 we see that $\psi \circ \phi$ is an isomorphism and
	 thus $q^{m}T(\pi)$ is a direct summand of $T$.
	 The complement $T/q^{m} T(\pi)$ is also tilting
	 by Theorem \ref{Thm:char}.
	 \end{proof}

	 \begin{Prop} \label{Prop:tiltres}
	 	 Let $H$ be an affine quasi-hereditary algebra
	with property $(\spadesuit)$.
	Then a module $V \in H \gmod$ belongs to $\F(\std)$
	if and only if $V$ has a finite right resolution by tilting modules.
	 \end{Prop}
	 \begin{proof}
	 Suppose that $V \in \F(\std).$
	 To prove that $V$ has a right tilting resolution,
	 we proceed by induction on
	 the size of the set $\overline{\supp(V)}$.
	 The case where $\supp(V) = \varnothing$
	 is trivial.
	 Let $\pi \in \overline{\supp(V)}$
	 be a maximal element.
	 By Lemma \ref{Lem:ext} (\ref{Lem:ext:std}),
	 we have a submodule
	 $U \in V$ which is a direct sum of
	 some grading shifts of $\std(\pi)$
	 and the quotient $V/U$ has 
	 a $\std$-filtration with
	 $(V/U : \std(\pi)) =0.$
	 Then we have $\overline{\supp(V/U)}
	 \subset \overline{\supp(V)} \setminus \{ \pi \}.$
	 By induction hypothesis,
	 we have a right tilting resolution
	$V/U \to T'_{0} \to T'_{1} \to \cdots$
	and get a map 
	$\phi : V \twoheadrightarrow V/U \to T'_{\bullet}.$
	On the other hand, Theorem
	\ref{Thm:tilt} shows that there is an exact sequence
	$0 \to \std(\pi) \to T(\pi ) \to Q(\pi) \to 0,$
	where $Q(\pi) \in \F(\std)$ and 
	$(Q(\pi):\std(\sigma)) = 0$ for any
	$\sigma \not < \pi $.
	Because the module $Q(\pi)$
	has a tilting resolution by the induction hypothesis,
	the standard module
	$\std(\pi)$ has a tilting resolution.
	Thus $U$ has a tilting resolution 
	$U \to T''_{0} \to T''_{1} \to \cdots.$
	Since $\ext_{H}^{1}(V/U, T''_{0}) = 0$ we can lift
	the map
	$U \to T''_{\bullet}$ to a map $\psi : V \to T''_{\bullet}$.
	Therefore we have a map
	$\phi \oplus \psi  : V \to T'_{\bullet} \oplus T''_{\bullet}$,
	which gives a tilting resolution of $V$.

	The other implication is a direct consequence
	of Theorem \ref{Thm:char} (\ref{Thm:char:std}).
	 \end{proof}

\subsection{A criterion for categorical equivalence}
\label{equiv}

	For an additive (resp.\ abelian)
	category $\Cc$, we denote its 
	bounded homotopy (resp.\ derived)
	category by
	$K^{b}(\Cc)$ (resp.\ $D^{b}(\Cc)$).
	
	\begin{Lem} \label{Lem:equiv}
	Let $H$ be an affine quasi-hereditary algebra
	with property $(\spadesuit)$
	and $\CC := H \gmod.$
	Then the natural functor
	$K^{b}(\T) \to D^{b}(\CC)$,
	which is the composition of
	the natural embedding 
	$K^{b}(\T) \hookrightarrow K^{b}(\CC)$
	and the quotient $K^{b}(\CC)\to D^{b}(\CC)$,
	gives an equivalence of graded triangulated categories. 
	\end{Lem}
	\begin{proof}
	Note that every affine quasi-hereditary algebra 
	has a finite global dimension by Remark
	\ref{Rem:gldim}.
	Thanks to Theorem
	\ref{Thm:char}, \ref{Thm:tilt} and 
	Proposition \ref{Prop:tiltres},
	we can apply the graded version of
	\cite{Hap} Lemma 1.1 and 1.5 to 
	the module $\bigoplus_{\pi \in \Pi}T(\pi).$  
	\end{proof}
	
	For each $i=1,2$,
	let $H_{i}$ be an affine quasi-hereditary algebra 
	with property $(\spadesuit)$ and 
	$\CC_{i} := H_{i} \gmod$ be 
	the associated affine highest weight category,
	whose index poset is denoted by $\Pi_{i}$. 
	For each $\pi \in \Pi_{i}$,
	we denote by $\std_{i}(\pi)$, (resp.$\pcstd_{i}(\pi), 
	T_{i}(\pi)$)
	the corresponding 
	standard (resp. proper costandard,
	indecomposable tilting) module in $\CC_{i}$.
	We denote by $\T_{i} := \F(\std_{i}) \cap \F(\pcstd_{i})$
	the additive full subcategory of tilting modules.

	\begin{Thm} \label{Thm:equiv}
	Under the above notation,
	assume that there is a graded exact functor $\Phi :
	\CC_{1} \to \CC_{2}$ and 
	a bijection $\phi: \Pi_{1} \to \Pi_{2}$ preserving 
	the partial orderings
	such that 
	we have $
	\Phi(\std_{1}(\pi)) \cong \std_{2}(\phi(\pi))$
	and
	$\Phi(\pcstd_{1}(\pi)) \cong \pcstd_{2}(\phi(\pi))
	$
	for all $\pi \in \Pi_{1}$.
	Then the functor $\Phi$
	gives an equivalence
	of graded categories
	$\Phi : \CC_{1 } \simeq \CC_{2}$.
	\end{Thm}
	First we prove the following lemma:

	\begin{Lem} \label{Lem:ff}
	Under the same assumption as in Theorem \ref{Thm:equiv},
	we have a natural isomorphism
	$\homo_{\CC_{1}}(U,V) \cong
	\homo_{\CC_{2}}(\Phi(U), \Phi(V)) $
	for any $U \in \F(\std_{1}), V \in \F(\pcstd_{1})$.
	In particular, we have a natural graded isomorphism
	$\eend_{\CC_{1}}(T) \cong \eend_{\CC_{2}}(\Phi(T))$
	for all $T \in \T_{1}.$
	\end{Lem}
	\begin{proof}
		Let $V = V_{0} \supset V_{1} \supset \cdots$
		be a $\pcstd_{1}$-filtration.
		By Theorem \ref{Thm:BHBGG} and
		the five lemma, we have that
		$\homo_{\CC_{1}}(U, V/V_{n})
		\cong \homo_{\CC_{2}}
		(\Phi(U), \Phi(V)/\Phi(V_{n})) $
		for all $n \geq 0$.
		Then Lemma \ref{Lem:NLC} (\ref{Lem:NLC:lim})
		completes the proof.
	\end{proof}

	\begin{proof}[Proof of Theorem \ref{Thm:equiv}]

	By assumption,
	the functor $\Phi$ sends a tilting module
	to a tilting module.
	By Theorem \ref{Thm:tilt},
	we have a direct sum decomposition:
	$\Phi(T_{1}(\pi)) \cong 
	\bigoplus_{\sigma \in \Pi_{1}} 
	T_{2}(\phi(\sigma))^{\oplus a_{\pi}^{\sigma}(q)} $
	for each $\pi \in \Pi_{1}$, 
	where $a_{\pi}^{\sigma}(q) \in \Z_{\geq 0}[q,q^{-1}]$.
	Since the bijection $\phi: \Pi_{1} \to \Pi_{2}$ preserves
	the partial orderings, we have
	$a_{\pi}^{\pi}(q) = 1$ and 
	$a^{\sigma}_{\pi}(q)=0$ for any $\sigma \not \leq \pi$.
	By Lemma \ref{Lem:ff},
	we have an isomorphism
	\begin{align*}
	\eend_{\CC_{1}}(T_{1}(\pi))
	& \cong \eend_{\CC_{2}}(\Phi(T_{1}(\pi))) \\
	&\cong \eend_{\CC_{2}}(T_{2}(\phi(\pi)))\oplus
	\bigoplus_{\sigma < \pi}
	\eend_{\CC_{2}}(
	T_{2}(\phi(\sigma))^{\oplus a^{\sigma}_{\pi}(q)})
	\oplus N,
	\end{align*}
	where $N$ is a part included in the radical.
	Using Theorem \ref{Thm:tilt} (\ref{Thm:tilt:local}),
	we observe that $a^{\sigma}_{\pi}(q) = 0$
	for all $\sigma < \pi$.
	
	This shows that $\Phi(T_{1}(\pi)) \cong T_{2}(\phi(\pi))$.
	By Theorem \ref{Thm:tilt} and Lemma \ref{Lem:ff},
	the functor $\Phi$ induces
	a graded equivalence  
	$\T_{1} \cong \T_{2}$ and therefore yields 
	an equivalence 
	$\Phi : K^{b}(\T_{1}) \simeq K^{b}(\T_{2})$
	of graded triangulated categories.
	By Lemma \ref{Lem:equiv}, we have
	$\Phi : D^{b}(\CC_{1}) \simeq D^{b}(\CC_{2})$.
	Since our functor $\Phi$ is exact,
	it yields
	an equivalence of the hearts of the
	standard $t$-structures, which are
	naturally identified with the original abelian categories
	$\CC_{i}$.
	\end{proof}

\section{Completion}
\label{Completion}

\subsection{Complete category}
\label{compcat}

	Let $\Cc$ be a $\kk$-linear abelian category with
	a subfunctor 
	$\varphi \hookrightarrow \mathrm{Id}_{\Cc}$
	which preserves epimorphisms.
	Note that each object $V \in \Cc$ has a filtration
	$V \supset \varphi(V) \supset \varphi^{2}(V) 
	\supset \cdots
	$.
	We say that the category $\Cc$ is
	 $\varphi$-adically complete if for each object
	 $V \in \Cc$,  the projective limit
	 $\varprojlim V/\varphi^{n}(V)$
	 exists and naturally isomorphic to
	 $V$.
	 An exact functor $\Phi : \Cc \to \Cc'$ from
	 a $\varphi$-adically complete category $\Cc$ 
	 to a $\varphi'$-adically complete category $\Cc'$
	 is said to be continuous if 
	 there is a positive integer $n$ such that
	 $\Phi \circ\varphi^{n}$
	 is a subfunctor of $\varphi' \circ \Phi.$
	 
	  Assume that $\Cc$ have a complete set
	   $\{L(\pi) \mid \pi \in \Pi\}$
	 of representatives of simple isomorphism classes in $\Cc$.
	 We consider the following property $(\clubsuit)$
	 for a $\varphi$-adically complete category $\Cc$:
	 \begin{itemize}
	 \item[$(\clubsuit)_{1}$]
	Each object $V \in \Cc$ is Noetherian;
	  \item[$(\clubsuit)_{2}$]
	 Each simple object $L(\pi)$
	 has a projective cover
	 $P(\pi) \twoheadrightarrow L(\pi)$;
	  \item[$(\clubsuit)_{3}$]
	  For each object $V \in \Cc$, the quotient
	  $V/\varphi(V)$ is of finite length.
	 \end{itemize}
	 
	 We give an example of such a category.
	 Let $H$ be a
	 left Noetherian $\kk$-algebra and $I \subset H$ be a
	 two-sided ideal  of finite codimension.
	 Assume that $H$ is
	 $I$-adically complete
	 i.e.\ $H \cong \varprojlim H/I^{n}.$
	 We define an endofunctor
	 $\varphi $ on $\Cc := H \mof$
	 by $\varphi(V) := IV$.
	Then $\varphi$ is naturally a subfunctor
	of $\mathrm{Id}_{\Cc}$
	preserving epimorphisms and
	$\Cc$ is a $\varphi$-adically
	 complete category with property $(\clubsuit).$ 

	  \begin{Prop} \label{Prop:compcat}
	  Let $\Cc$ be a $\varphi$-adically complete category
	  with property $(\clubsuit)$.
	  Assume that the set $\Pi$ is finite. 
	  We put
	  $P := \bigoplus_{\pi \in \Pi}P(\pi)$
	  and $H := \End_{\Cc}(P)^{\mathrm{op}}$.
	  Then
	  the algebra $H$ is left Noetherian, complete
	  with respect to the finite-codimensional ideal
	  $I:=\Hom_{\Cc}(P, \varphi(P))$.
	  Moreover the functor
	  $\Phi := \Hom_{\Cc}(P, -)$
	  gives a continuous equivalence
	  $\Phi : \Cc \simeq H \mof.$
	  \end{Prop}
	  \begin{proof}
	  By property $(\clubsuit)$,
	  the category $\Cc$ has enough projectives and
	  every projective object is isomorphic to
	  a finite direct sum of $P(\pi)$'s for some $\pi \in \Pi$.
	  Therefore $P$ is a projective generator of the category
	  $\Cc$ and the equivalence
	  $\Phi:\Cc \simeq H\mof$ follows from 
	  a general result of abelian categories
	  (see \cite{Bass} Theorem II.1.3).
	  The algebra $H$ is left Noetherian by $(\clubsuit)_{1}$.
	  We have to show that 
	  the ideal $I$ is finite-codimensional and 
	  the algebra
	  $H$ is $I$-adically complete.
	  First we claim that for any object $V \in \Cc$,
	  we have
	  $I \cdot \Hom_{\Cc}(P, V) = \Hom_{\Cc}(P, \varphi(V)).$
	  Because $P$ is a projective generator of $\Cc$,
	  we have a epimorphism
	  $a : P^{N} \twoheadrightarrow V$ for some $N>0.$
	  Since the functor $\varphi$ preserves epimorphisms, 
	  the morphism
	  $\varphi(a): \varphi(P)^{\oplus N} \to \varphi(V)$
	  is still epic and induces a surjection
	  $\Hom_{\Cc}(P, \varphi(P))^{\oplus N} = I^{\oplus N}
	  \twoheadrightarrow
	  \Hom_{\Cc}(P, \varphi(V) ),$
	  from which the claim follows.
	Then we have
	$\Hom_{\Cc}(P, P/\varphi^{n}(P))
	\cong \Hom_{\Cc}(P,P)/\Hom_{\Cc}(P, \varphi^{n}(P))
	= \End_{\Cc}(P)/I^{n}.$
	Therefore $I$ is finite-codimensional by $(\clubsuit)_{3}.$
	Because the category $\Cc$ is 
	$\varphi$-adically complete,
	we have $
	\End_{\Cc}(P)
	= \varprojlim \Hom_{\Cc}(P, P/\varphi^{n}(P))
	= \varprojlim \End_{\Cc}(P)/I^{n},
	$
	which proves that the algebra $H$ is $I$-adically complete.
	  \end{proof}

\subsection{Complete affine highest weight category}
\label{completeaffine}

	Let $\Cc$ be a $\varphi$-adically complete category
	with property $(\clubsuit)$.
	We write $\hhd V$ (resp.\ $\rrad V$) to indicate
	the head (resp.\ radical) 
	of $V \in \Cc$.
	We assume that the
	index set $\Pi$ is equipped with a partial order $\leq$.
	Then we define the standard object $\std(\pi)$ and
	the proper standard object $\pstd(\pi)$
	for each $\pi \in \Pi$ as:
	\begin{align*}
	  \std(\pi) &:= P(\pi)/ (
	  \sum_{\sigma \not \leq 
	  \pi, f \in \Hom_{\Cc}(P(\sigma), P(\pi))}
	  \Ima f ),\\
	\pstd(\pi) &:= P(\pi) / (
	  \sum_{\sigma \not < \pi, f \in \Hom_{\Cc}(P(\sigma), 
	  \rrad P(\pi))}
	  \Ima f ).
	\end{align*}

	\begin{Def} \label{Def:comphwc}
	Let $\Cc$ be a $\varphi$-adically complete category
	with property $(\clubsuit)$.
	The category $\Cc$
	is called a complete affine highest weight category
	if the following three conditions hold:
	\begin{enumerate}
	\item \label{Def:comphwc:str}
	For each $\pi \in \Pi$,
	the kernel of the natural quotient morphism
	$P(\pi) \twoheadrightarrow \std(\pi)$
	has a $\std$-filtration whose subquotients are 
	$ \cong \std(\sigma)$ for some $\sigma > \pi$;

	\item \label{Def:comphwc:end}
	For each $\pi \in \Pi$, the algebra
	$B_{\pi} := \End_{\Cc}(\std(\pi))$
	is isomorphic to a ring
	$\kk [ \! [ z_{1}, \ldots , z_{n_{\pi}}] \! ]$
	of formal power series in $n_{\pi}$-variables
	for some $n_{\pi} \in \Z_{\geq 0}$;

	\item \label{Def:comphwc:free}
	For any $\pi, \sigma \in \Pi$,
	the $B_{\sigma}$-module
	$\Hom_{\Cc}(P(\pi), \std(\sigma))$
	is free of finite rank. 
	\end{enumerate}
	\end{Def}

	In the case of complete category,
	we also have the parallel theory
	of affine highest weight category
	to the case of graded category as 
	in Section \ref{affhwc} and 
	\ref{tilting},
	although we should replace the property 
	$(\spadesuit)$ with the following $(\hat{\spadesuit})$:
	\begin{itemize}
		\item[$(\hat{\spadesuit})$]
		There is a central local subalgebra $Z \subset H$
		with maximal ideal $\m$ such that
		$Z/\m \cong \kk$ and $H$
		is finitely generated as a $Z$-module.
	 \end{itemize}
	 Note that this is equivalent to simply saying that
	 the algebra $H$ is a finitely generated module 
	 over its center. 
	 
	We omit the other precise statements and their proofs
	since they are quite similar to those
	of graded cases.

\subsection{Completion functor}
\label{completionfunc}

	For a graded $\kk$-vector space $V$,
	we write $V^{\mathtt{f}}$ to indicate that 
	we forget its grading structure.
	For a graded $\kk$-algebra $H$, 
	we have the forgetful functor
	$(-)^{\mathtt{f}} :  H \gmod \to H^{\mathtt{f}} \mof
	; V \mapsto V^{\mathtt{f}}.$

	For a graded $\kk$-vector space 
	$V = \bigoplus_{n \in \Z}V_{n}$, 
	we define its formal completion by
	$\hat{V} := \prod_{n \in \Z}V_{n}.$
	If $H$ is a Laurentian graded algebra,
	$\hat{H}$ naturally becomes a $\kk$-algebra
	and we have a natural 
	isomorphism $\hat{H} \cong \varprojlim
	H^{\mathtt{f}}/I^{n}$
	for any finite-codimensional ideal $I 
	\subset (H_{>0})^{\mathtt{f}}.$
	In this case, the completion $\hat{V}$ of 
	a module $V \in H \gmod$ naturally 
	becomes  a module over $\hat{H}$.

	 \begin{Lem} \label{Lem:compl1}
	 Let $H$ be a Noetherian Laurentian algebra
	with property $(\spadesuit)$.
	\begin{enumerate}
	\item
	  The completion $\hat{H}$ is left Noetherian 
	and complete with respect to 
	the 
	finite-codimensional 
	ideal $\hat{\m}:= \m \hat{H}= \hat{H}\m$;
	\item
	We have $\hat{V} \cong 
	\varprojlim V^{\mathtt{f}}/\m^{n}V^{\mathtt{f}} 
	\cong \hat{H}\otimes_{H}V^{\mathtt{f}}$
	for any $V \in H \gmod$;
	\item
	The completion functor  
	$H \gmod \to \hat{H} \mof; \, V \to \hat{V}$
	is exact.
	\end{enumerate}
	 \end{Lem}
	 \begin{proof}
	 See
	  \cite{NvO} Section D.V and 
	 the above observations. 
	 \end{proof}

	 \begin{Lem} \label{Lem:compl2}
	 Let $H$ be a Noetherian Laurentian algebra
	with property $(\spadesuit)$.
	 Then for any $U, V \in H \gmod$, we have
	 a natural isomorphism
	 $ \Hom_{\hat{H}}(\hat{U}, \hat{V})
	 \cong \prod_{n \in \Z} \homo_{H}(U, V)_{n} .
	 $
	 In particular, the completion functor 
	 $H \gmod \to \hat{H} \mof$ 
	 is faithful.
	 \end{Lem}
	 \begin{proof}
	 If $V$ is of finite length, we have
	 $$\Hom_{\hat{H}}(\hat{U}, \hat{V})
	 \cong \Hom_{H^{\mathtt{f}}} (U^{\mathtt{f}}, 
	 V^{\mathtt{f}})
	 \cong \homo_{H}(U, V)^{\mathtt{f}}.$$
	 The third term is equal to 
	 $\prod_{n \in \Z}\homo_{H}(U, V)_{n}$
	 because
	 only finitely many direct summands 
	 are non-zero. 
	 In general,
	 note that we have
	 $\hat{V}/ \hat{\m}^{n} \hat{V} \cong
	 V^{\mathtt{f}}/\m^{n}V^{\mathtt{f}} $ for all $n >0$ 
	 and we have:
	 $$
	 \Hom_{\hat{H}}(\hat{U}, \hat{V})
	 \cong \varprojlim
	 \Hom_{\hat{H}}(\hat{U},
	  \hat{V}/\hat{\m}^{n}\hat{V})
	 \cong
	 \varprojlim
	 \homo_{H}(U, V/\m^{n}V)^{\mathtt{f}}.
	 $$
	 By Lemma \ref{Lem:NLC} (\ref{Lem:NLC:fin}), 
	 there is
	 an increasing sequence of integers 
	 $N_{1} < N_{2} < \cdots$
	 such that
	 $\homo_{H}(U, \m^{n}V )_{m} 
	 =\ext_{H}^{1}(U, \m^{n}V)_{m}
	 = 0$ for all $m < N_{n}$.
	Then we have a natural short 
	exact sequence of projective systems:
	$
	0 \to
	\{ \bigoplus_{m < N_{n}}\homo_{H}(U, V)_{m} \}_{n}
	\to
	\{ \bigoplus_{m \in \Z}\homo_{H}(U, V/\m^{n}V)_{m}  \}_{n}
	\to
	\{
	\bigoplus_{m \geq  N_{n}}\homo_{H}(U, V/\m^{n}V)_{m}
	\}_{n}
	\to 0.
	$
	By taking the projective limits, we get
	$\varprojlim
	\bigoplus_{m \geq  N_{n}}\homo_{H}(U, V/\m^{n}V)_{m}
	= 0
	$
	and therefore we have:
	$$
	\prod_{m \in \Z}\homo_{H}(U, V)_{m}
	\cong
	\varprojlim_{n} \bigoplus_{m < N_{n}}
	\homo_{H}(U, V)_{m}
	\cong
	\varprojlim_{n}
	\bigoplus_{m \in \Z}\homo_{H}(U, V/\m^{n}V)_{m}.
	$$
	 \end{proof}

	\begin{Lem} \label{Lem:compl3}
	Let $H$ be a Noetherian Laurentian algebra
	with property $(\spadesuit)$.
	 Let $\{ L(\pi) \mid \pi \in \Pi \}$
	 be a complete and irredundant set
	 of simple modules
	 in $H \gmod$
	 up to isomorphism and grading shift.
	 Then:
	 \begin{enumerate}

	 \item \label{Lem:compl3:simple}
	 the set
	 $\{ L(\pi)^{\wedge}  \in \hat{H} \mof 
	 \mid \pi \in \Pi \}$
	 is a complete set of isomorphism classes of
	 simple modules in $\hat{H}\mof$;

	 \item \label{Lem:compl3:proj}
	 the $\hat{H}$-module $P(\pi)^{\wedge}$
	 is a projective cover of the simple module
	 $L(\pi)^{\wedge}$ for each $\pi \in \Pi$.
	 \end{enumerate}
	\end{Lem}
	\begin{proof}
	(\ref{Lem:compl3:simple})
	Let $L \in \hat{H}\mof$ be a simple module.
	By Nakayama's lemma,
	we have $\m L=0 $.
	Then $L$
	 is gradable by \cite{Kash} Lemma 5.1.6.

	(\ref{Lem:compl3:proj})
	Let $P \in H \gmod$ be a projective module.
	Then 
	we have 
	$\Ext_{\hat{H}}^{1}(\hat{P}, L(\pi)^{\wedge} ) =
	\ext_{H}^{1}(P, L(\pi))
	= 0$
	for all $\pi \in \Pi$.
	Then for any $V \in \hat{H}\mof$, we have
	$\Ext^{1}_{\hat{H}}(\hat{P}, V) = 0$
	by Lemma \ref{Lem:limvan}.
	Therefore $\hat{P}$ is projective
	in $\hat{H} \mof.$
	Moreover, by Lemma \ref{Lem:compl2}, we have:
	$$\Hom_{\hat{H}}(P(\pi)^{\wedge}, L(\sigma)^{\wedge}) =
	\begin{cases}
	\kk & \sigma = \pi; \\
	0 & \sigma \neq \pi.
	\end{cases}
	$$
	This shows that $P(\pi)^{\wedge}$ is a projective cover of
	$L(\pi)^{\wedge}$ in $\hat{H}\mof$.
	\end{proof}

	Combining the above lemmas,
	we conclude as follows.

	\begin{Thm} \label{Thm:compl}
	Let $H$ be a graded affine quasi-hereditary algebra
	with property $(\spadesuit)$. 
	Then $\hat{H}$ is a complete
	affine quasi-hereditary algebra 
	satisfying the property
	$(\hat{\spadesuit})$
	with respect to $\hat{Z}$.
	The standard
	(resp. proper standard, proper costandard,
	indecomposable tilting) module
	in $\hat{H}\mof$
	associated to $\pi \in \Pi$ is
	the completion of the 
	counterpart
	in $H \gmod$.  
	Moreover the complete affine algebra
	$\End_{\hat{H}}(\std(\pi)^{\wedge})$
	is equal to the formal completion 
	$\hat{B}_{\pi}$
	of
	the graded affine algebra
	$B_{\pi} = \eend_{H}(\std(\pi))$.
	\end{Thm}

\section{An application: The Arakawa-Suzuki functor}
\label{application}

\subsection{The degenerate affine Hecke algebra of $GL_n$}
\label{dAHA}
		We fix a positive integer $n > 0$.

		\begin{Def} \label{Def:dAHA}
		The degenerate affine Hecke algebra
		of $GL_n$ is the $\C$-algebra $H_{n}$ defined 
		by the generators
		$$\{ x_1, \ldots x_{n}\} \cup \{ s_1, \ldots , s_{n-1}\}$$
		subject to the following relations
		\begin{alignat*}{2}
			x_i x_j &= x_j x_i , &
			s_i^{2} &= 1, \\
			s_i s_{i+1} s_{i} &= s_{i+1} s_i s_{i+1} , \qquad&
			s_i s_j &= s_j s_i \quad \text{if} \, \, |i-j| > 1 ,\\
			x_{i+1}s_i &= s_i x_i +1, &
			x_j s_i &= s_i x_j \quad \text{if} \, \,  j \neq i, i+1.
		\end{alignat*}
		\end{Def}
		
	Let $P_{n} := \C[x_1, \ldots, x_n]$ be 
	the polynomial ring
	and 
	let $\C \SG_{n}$ be
	the group algebra 
	of the symmetric group of degree $n$
	which is generated by the simple reflections 
	$\{s_{1} , \ldots, s_{n-1} \}$.
	From Definition \ref{dAHA}, we have 
	the natural $\C$-algebra homomorphisms
	$P_{n} \to H_{n}$ and $\C \SG_{n} \to H_{n}$,
	which are injective by 
	Proposition \ref{Prop:dAHA} below.
		
	Note that there is an anti-involutive algebra homomorphism
	$\psi : H_{n} \to H_{n}$ which
	fixes all the generators $x_{i}, s_{i}$.
	 For a left $H_{n}$-module $V$,
	  we define a left $H_{n}$-module structure
	  on the dual space $V^{\circledast} := \Hom_{\C}(V, \C)$
	  by twisting the natural right $H_{n}$-action
	  by the anti-involution $\psi$.
	
	\begin{Prop}[cf.\ \cite{Kbook} Theorem 3.2.2 and 3.3.1]
	 \label{Prop:dAHA}
	\ 
	  \begin{enumerate}
	  \item \label{Prop:dAHA:PBW}
	  As a $(\C \SG_{n}, P_{n})$-bimodule, we have
	  $H_{n} \cong \C \SG_{n} \otimes_{\C} P_{n}$;
	  \item  \label{Prop:dAHA:center}
	  The center $Z(H_{n})$ of $H_{n}$ is equal to
	  the subalgebra of symmetric polynomials in $P_{n}.$ I.e.\ 
	  $Z(H_{n}) = (P_{n})^{\SG_{n}}$.
	  \end{enumerate}
	  \end{Prop}
	Let $Q:= \bigoplus_{i \in \Z}\Z \alpha_{i}$
	be the root lattice of type $A_{\infty}$ and
	define $Q^{+} := \sum_{i \in \Z} \Z_{\geq 0} \alpha_{i}$. 
	For each $i,j \in \Z$ with
	 $i < j$, 
	we define the positive root $\alpha(i,j)$ by
	$\alpha(i,j):= \alpha_{i} + \alpha_{i+1} + 
	\cdots + \alpha_{j-1}
	\in Q^{+}$.
	For each element $\beta = \sum_{i \in \Z} n_{i}\alpha_{i}
	\in Q^{+}$, we define its height 
	by $\mathtt{ht}(\beta):= \sum_{i \in \Z} n_{i}$.
	
	We fix an 
	element $\beta \in Q^{+}$ with $\mathtt{ht}(\beta) = n$.
	We regard $\beta$ as an unordered
	set of $n$ integers with multiplicity i.e.\ 
	$\beta \in \Z^{n}/\SG_{n} \subset \C^{n}/ \SG_{n}
	=
	\mathrm{Specm}(Z(H_{n})),$
	where the equality is due to
	Proposition \ref{Prop:dAHA} (\ref{Prop:dAHA:center}).
	We denote the formal completion
	of the affine algebra $Z(H_{n})$
	at $\beta$ by
	$\hat{Z}(H_{n})_{\beta}$ and define
	$\hat{H}_{\beta} 
	:= H_{n}\otimes_{Z(H_{n})} \hat{Z}(H_{n})_{\beta}.$

	  An unordered multiset
	  $\pi = \{ \alpha(i_{1},j_{1}), \ldots,
	  \alpha(i_{m},j_{m}) \}$
	  of positive roots
	  such that
	  $\alpha(i_{1},j_{1}) + \cdots +
	  \alpha(i_{m},j_{m}) = \beta$
	  is called
	  a Kostant partition
	  of $\beta$.
	  We denote by $\KP(\beta)$
	  the set of all Kostant partitions of $\beta$.
	  This is a finite set.
	  Let $\pi, \pi' \in \KP(\beta).$
	 We write $\pi \unlhd \pi'$ if $\pi = \pi'$
	 or one of the following two conditions hold:
	  \begin{enumerate}
	  \item
	  There are integers $i, j, k$ with $i < j <k$ such that
	  $\alpha(i,k) \in \pi$ and $\alpha(i,j), \alpha(j,k) \in \pi'$
	  with 
	  $\pi \setminus \{ \alpha(i,k) \} = \pi' \setminus \{\alpha(i,j)
	  , \alpha(j,k)\}$;
	  \item
	  There are integers $i, j ,k ,l$ with $i < j < k <l$ such that
	  $\alpha(i,l), \alpha(j,k) \in \pi$ and
	  $\alpha(i,k), \alpha(j,l) \in \pi'$
	  with
	  $\pi \setminus \{\alpha(i,l), \alpha(j,k)\}
	  = \pi' \setminus \{ \alpha(i,k), \alpha(j,l)\}$.
	  \end{enumerate}
	  The relation $\unlhd$ generates a
	  partial order on $\KP(\beta)$,
	  which we denote by the same symbol $\unlhd$.
	  
	\begin{Thm} \label{Thm:dAHA}
	The algebra
	$\hat{H}_{\beta}$ is
	a complete affine quasi-hereditary algebra 
	with the index poset $(\KP(\beta), \unlhd)$
	satisfying the property $(\hat{\spadesuit})$.
	\end{Thm}
	\begin{proof}
	Brundan-Kleshchev-McNamara \cite{BKM}
	and Kato \cite{Kato} proved that 
	the KLR algebra $R_{\beta}$
	associated to $\beta$ 
	is a graded affine quasi-hereditary algebra
	with the index poset $(\KP(\beta), \unlhd)$
	(see Remark \ref{Rem:order} below).
	By Khovanov-Lauda \cite{KL} Corollary 2.10,
	we can easily see that $R_{\beta}$ satisfies the property 
	$(\spadesuit).$
	Then the completion
	$\hat{R}_{\beta}$
	is a complete affine quasi-hereditary algebra
	with property $(\hat{\spadesuit})$
	by Theorem \ref{Thm:compl}.
	Now we complete the proof by the 
	isomorphism
	$\hat{R}_{\beta} \cong \hat{H}_{\beta}$
	of topological $\C$-algebras proved 
	by Brundan-Kleshchev \cite{BK}.
	\end{proof}
	
	\begin{Rem} \label{Rem:order}
	By \cite{ext},
	a partial order for
	an affine quasi-hereditary structure of $R_{\beta}$
	comes from  
	the orbit closure relation on 
	the space
	$\mathfrak{L}(\beta)
	:= \bigoplus_{i \in \Z} \Hom_{\C}(\C^{n_{i}}, \C^{n_{i+1}})
	$ of representations of linear quiver
	of dimension vector $\beta$
	with respect to the action of the group
	$G(\beta) := \prod_{i \in \Z}GL(\C^{n_{i}}).$
	The
	closure relations
	of $G(\beta)$-orbits in $\mathfrak{L}(\beta)$
	has been described  in Abeasis-DelFra-Kraft
	\cite{ADK} Proposition 9.1,
	which asserts that 
	the set $\KP(\beta)$ naturally 
	parametrizes the $G(\beta)$-orbits in
	$\mathfrak{L}(\beta)$
	and  
	our ordering
	$\unlhd$ on $\KP(\beta)$
	corresponds to the opposite 
	closure ordering of orbits.  
	\end{Rem}

	We denote the standard (resp. proper standard,
	proper costandard) module of $\hat{H}_{\beta} \mof$
	associated to $\pi \in \KP(\beta)$
	by $\std_{H}(\pi)$ (resp.\ 
	$\pstd_{H}(\pi), \pcstd_{H}(\pi)$).

	\begin{Rem} \label{Rem:stdH}
	We define a total order $\preceq$ on the set
	of positive roots so that
	  $\alpha(i,j) \preceq \alpha(k,l)$
	  if $i < k$, or $i=k$ and $j \leq l$.
	For each
	$\pi = \{ \pi_{1}, \ldots , \pi_{m} \} \in \KP(\beta)$
	such that $\pi_{1} \succ \pi_{2} \succ
	\cdots \succ \pi_{m}$,
	we give an explicit construction of the standard module
	$\std_{H}(\pi)$ following \cite{BKM}
	(see also \cite{KR} Section 8.4). 
	Let $\pi_{k} = \alpha(i_k, j_k)$ and set
	$n_{k} := j_{k} - i_{k}>0,
	a_{k} := \sum_{l=1}^{k-1} n_{l} +1
	$
	for $1 \leq k \leq m.$
	Then we define an 
	$H_{n_1}\boxtimes \cdots \boxtimes H_{n_m}$-module
	structure on the space $R_{m} :=
	\C[\![z_{1}, \ldots, z_{m}]\!]$ so that
	$x_{l}$ acts on by the multiplication of the scalar
	$i_{k} + (l-a_{k}) + z_{k}$
	if $a_{k} \leq l < a_{k+1}$
	and the action of 
	$\SG_{n_{1}} \times \cdots \times \SG_{n_{m}}$
	is trivial.
	Then we define $\std_{H}(\pi)$ as
	the induced module
	$H_{n} \otimes_{(H_{n_{1}}
	\boxtimes\cdots \boxtimes H_{n_{m}})}
	R_{m}$,
	which
	naturally extends to a module over the
	completion $\hat{H}_{\beta}$.
	By construction, $\std_{H}(\pi)$ has a
	natural action of $R_{m}$
	which commutes with the action of $\hat{H}_{\beta}$.
	Then we have
	$\pstd_{H}(\pi) = \std_{H}(\pi)/\m \std_{H}(\pi)$,
	where $\m$ is the maximal ideal of $R_{m}$,
	and $\pcstd_{H}(\pi) = \pstd_{H}(\pi)^{\circledast}.$
	\end{Rem}

\subsection{The deformed BGG category of $\gl_{m}$}
\label{defBGG}

	Let $m \, (\geq 2)$ be a positive integer and
 	$\g := \gl_{m}(\C) = \mathrm{Mat}_{m}(\C)$ 
 	be the general linear Lie algebra.
	Let $e_{ij} \in \g$ denote the $ij$-matrix unit 
	$(1 \leq i, j \leq m)$.
	Let $\ct$ be the abelian Lie subalgebra of $\g$ consisting of
	diagonal matrices, i.e.\ $\ct = \bigoplus_{i=1}^{m}\C e_{ii}$.
	We have the natural pairing 
	$\langle - , - \rangle : \ct^{*} \times \ct \to \C$
	and identify the space  $\ct^{*}$ with $\C^m$
	by $\ct^{*} \stackrel{\sim}{\to} \C^{m} ; \,
	 \lambda \mapsto
	(\lambda_{1} , \ldots , \lambda_{m}) $
	where $\lambda_{i } = \langle \lambda, e_{ii} \rangle$.
	Note that the symmetric group
	$\SG_{m}$ acts on the space $\ct^{*} = \C^{m}$
	by permuting the coordinates.
	Let $\varepsilon_i \in \ct^{*}$  be the element defined by
	$\langle \varepsilon_i , e_{jj} \rangle = \delta_{ij}$.
	We define the weight lattice by
	$\Lambda := \bigoplus_{i=1}^{m} \Z \varepsilon_{i}$
	and the set of dominant weights
	by
	$\Lambda^{+} 
	:=\{\lambda \in \Lambda \mid \lambda_{1}
	\geq \lambda_{2} \geq \cdots \geq \lambda_{m} \}.$
	Let $\Lambda_{r} := \bigoplus_{i=1}^{m-1}
	\Z (\varepsilon_{i} - \varepsilon_{i+1}) \subset \ct^{*}$
	be the root lattice and
	$\Lambda_{r}^{+} := \sum_{i=1}^{m-1}\Z_{\geq 0}
	(\varepsilon_{i} - \varepsilon_{i+1}) $.
	We define a partial order called the dominance order $\leq$ 
	on $\Lambda$ so that 
	$\lambda \leq \mu$
	if $\mu - \lambda \in \Lambda_{r}^{+}.$
	We set a special dominant weight 
	$\rho$ by $ \rho := (0, -1, -2, \ldots, -m+1).  $
	Let $\n_{+}$ (resp.\ $\n_{-}$) be the nilpotent radical
	of upper (resp.\ lower) triangular matrices.
	We set the standard Borel subalgebra 
	$\be := \ct \oplus \n_{+}$.

	We identify the affine coordinate ring $
	\C[\ct^{*}] = \mathrm{Sym}(\ct)$
	with the polynomial ring 
	$\C [z_1, \ldots z_m]$ of $m$-variables
	by setting $z_i = e_{ii} \in \ct  \subset \C[\ct^{*}].$
	Let 
	$R := \C[\![z_1, \ldots , z_m]\!]$ be
	the completion of $\C[\ct^{*}]$,
	which is a local $\C$-algebra
	with its maximal ideal
	$\m := (z_{1} , \ldots, z_{m})$.

	\begin{Def} \label{Def:Otil}
		We define the 
		deformed BGG
		category $\Otil$ of $\g$ as 
		a full subcategory of the category
		$(U(\g) \boxtimes_{\C}R) \text{-Mod}$
		of all $(\g, R)$-bimodules.
		A $(\g, R)$-bimodule $M$ belongs to $\Otil$
		if the following three conditions hold:
		\begin{enumerate}
		\item $M$ is finitely generated as a $(\g, R)$-bimodule;
		\label{Def:Otil:fg}
		\item $M$ is locally finite over $U(\n_{+})$;
		\label{Def:Otil:ln}
		\item As a $(U(\ct), R)$-bimodule,
		\label{Def:Otil:ss}
		$M$ has a direct sum decomposition
		$ M = \bigoplus_{\lambda \in  \Lambda} 
		M_{(\lambda)} ,$
		where
		$M_{(\lambda)} := 
		\{v \in M \mid (e_{ii}-\lambda_{i}) v = v z_{i},
		\, (1 \leq i \leq m) \}$.
			\end{enumerate}  
	\end{Def}
	We refer the non-zero $(U(\ct), R)$-subbimodule
	$M_{(\lambda)}$ as a
	generalized weight space of weight $\lambda$.
	If $M_{(\lambda)} \neq 0$, the element 
	$\lambda$ is called
	an weight of the module $M$.
	Note that in Definition \ref{Def:Otil} (\ref{Def:Otil:ss})
	we restrict all the weights
	to be integral (i.e.
	belong to $\Lambda$) for simplicity. 
	Since the algebra $U(\g) \boxtimes_{\C} R$ is
	left Noetherian, the category $\Otil$ is a left Noetherian
	abelian category.
	The full subcategory of $\Otil$
	on which the $R$-action is trivial   
	is nothing but
	the usual BGG category $\OO$ of $\g$.
	We define a functor $\varphi
	: \Otil \to \Otil$
	by $\varphi(M) := M\m .$
	Note that this functor $\varphi$ is a subfunctor
	of identity functor $\mathrm{Id}_{\Otil}$
	which preserves epimorphisms.
	By 
	Definition \ref{Def:Otil} (\ref{Def:Otil:ss}),
	we see that
	$M \cong \varprojlim M/M\m^{n}$
	for all $M \in \Otil$.
	Thus the category $\Otil$ is a
	$\varphi$-adically complete category
	in the sense of Subsection \ref{compcat}.
	Because the usual category $\OO$ is Artinian,
	the specialized module $M/M\m$ is of finite length for any
	$M \in \Otil.$

	For each $\lambda \in \Lambda$,
	we define the deformed Verma module by
	$$ \Mtil(\lambda)
	:= U(\g) \otimes _{U(\be)} R_{\lambda - \rho},$$
	where $R_{\mu}$ denotes the 
	$(U(\be), R) $-bimodule defined as follows.
	As a right $R$-module,
	we have $R_{\mu} \cong R $.
	The left action of $U(\be)$
	on $R_{\mu}$ is defined by
	\begin{alignat*}{2}
	e_{ii} v & = v ( \mu_{i} + z_i) \qquad  
	&(1 \leq i \leq m);& \\
	e_{ij} v &= 0  \qquad  &( 1 \leq  i < j \leq m),&
	\end{alignat*}
	where $v \in R_{\mu}.$
	The deformed Verma module
	$\Mtil(\lambda)$ belongs to 
	$\Otil$ for any $\lambda \in \Lambda$.
	Note that the usual Verma module $M(\lambda)$
	is identical to the minimal specialization
	$\Mtil(\lambda)/\Mtil(\lambda) \m$ and 
	the generalized weight space 
	$\Mtil(\lambda)_{(\mu)}$ is a free $R$-module
	of rank $\dim_{\C}M(\lambda)_{(\mu)}$ 
	for each $\mu \in \Lambda.$
	In particular,
	we have $\hhd\Mtil(\lambda) = 
	\hhd M(\lambda) = L(\lambda)$,
	where $L(\lambda)$ is the unique irreducible
	quotient of $M(\lambda)$.
	The set $\{ L(\lambda) \mid \lambda \in \Lambda \}$
	is a complete set of representatives
	of irreducible isomorphism classes of the category $\Otil.$

	For each $\lambda \in \Lambda^{+}$, we define
	the category
	$\Otil_{\lambda}$
	to be the full subcategory
	of $\Otil$
	consisting of modules
	any of whose composition factors are 
	$\simeq L(w\lambda)$
	for some $w \in \SG_{m}.$
	For $\lambda \in \Lambda^{+}$,
	we define the subset $\Pi_{\lambda} $ of $\Lambda$ 
	by $ \Pi_{\lambda} :=
	\{w\lambda \in \Lambda \mid w \in \SG_{m}\}$
	and regard it as a subposet.
	Note that the dominant weight $\lambda$
	is the largest element in the poset $\Pi_{\lambda}$.

	\begin{Prop}[Soergel \cite{HC}, Fiebig \cite{F03}]
	 \label{Prop:Otildecomp}
	We have a decomposition of $\varphi$-adically
	complete category:
	$
	\Otil = \bigoplus_{\lambda \in \Lambda^{+}} 
	\Otil_{\lambda}.
	$
	\end{Prop}
	
	We call the category
	$\Otil_{\lambda}$ the block of the category $\Otil$
	associated to the dominant weight $\lambda$.

	\begin{Thm}[Soergel, Fiebig]
	 \label{Thm:Otil}
	Let $\lambda $ be a dominant weight.
	\begin{enumerate}
	\item \label{Thm:Otil:Verma}
	For each weight $\mu \in \Pi_{\lambda}$,
	the deformed Verma module $\Mtil(\mu)$
	belongs to $\Otil_{\lambda}$.
	The set $\{ L(\mu) \mid \mu \in \Pi_{\lambda}\}$
	is the complete set of representatives of
	irreducible isomorphism classes of $\Otil_{\lambda}$;

	\item \label{Thm:Otil:proj}
	The category $\Otil$ has enough projectives.
	For each $\mu \in \Pi_{\lambda}$,
	there exists a projective cover $P(\mu)$
	of the simple module $L(\mu)$ 
	in the block $\Otil_{\lambda}$.
	It fits into an exact sequence
	$
	0 \to K(\mu) \to P(\mu) \to \Mtil(\mu) \to 0,
	$
	where $K(\mu)$ has a filtration
	whose subquotients are $\simeq \Mtil(\nu)$ for some
	$\nu \in \Pi_{\lambda}$ with $\nu > \mu;$

	\item \label{Thm:Otil:free}
	The $R$-module
	$\Hom_{\Otil}(P(\mu), \Mtil(\nu))$
	is free of finite rank
	for any $\mu, \nu \in \Pi_{\lambda}$.
	\end{enumerate}
	\end{Thm}
	\begin{proof}
	See Soergel
	\cite{KategorieO} Theorem 6 or
	Fiebig
	\cite{F03} Theorem 2.7.
	(\ref{Thm:Otil:free}) is due to
	Soergel \cite{HC} Theorem 5.
	\end{proof}

	Let $\widetilde{\mathcal{M}}$ be the full subcategory
	of $(U(\g) \boxtimes_{\C} R) \text{-Mod}$ 
	consisting of modules $M$
	having a generalized weight space decomposition
	$M := \bigoplus_{\lambda \in \Lambda}M_{(\lambda)}$
	with each generalized weight space not necessary
	being finitely generated over $R$.
	For each $M \in \widetilde{\mathcal{M}}$,
	we define $\g$-module structures
	on 
	its restricted dual $M^{\vee} :=
	 \bigoplus_{\lambda \in \Lambda} 
	\Hom_{\C}(M_{(\lambda)}, \C)$ and
	its restricted topological dual
	$D(M) := \bigoplus_{\lambda \in \Lambda} 
	\Fun(M_{(\lambda)}, \C)$
	by twisting the natural right $\g$-action 
	by the transpose map,
	where we define $\displaystyle \Fun(N, \C)
	:= \varinjlim \Hom_{\C}(N/\m^{n}N, \C)$
	for an $R$-module $N$.
	Then we have two contravariant
	endofunctors $(-)^{\vee}$ and $D$
	satisfying that $(M^{\vee})_{(\lambda)} = 
	\Hom_{\C}(M_{(\lambda)}, \C)$ 
	and $D(M)_{(\lambda)}
	= \Fun(M_{(\lambda)}, \C).$
	Note that the functor $(-)^{\vee}$ is exact
	and the functor $D$ is left exact.
	If all generalized weight spaces of
	$M \in \widetilde{\mathcal{M}}$
	are free $R$-modules of finite rank,
	then we have $\mathbf{R}^{1}D(M)=0$
	and $(D(M))^{\vee} \cong M.$  
	Although these
	functors do not preserve the category $\Otil$,
	they preserve the category $\OO$,
	on which we have $D = (-)^{\vee}$
	and they are involutive.
	We have $L^{\vee} \cong L$ 
	for any irreducible module $L \in \OO$.
	For each $\lambda \in \Lambda$, 
	we define the dual Verma module
	to be $M(\lambda)^{\vee}.$
	
	\begin{Prop} \label{Prop:BGGrecip}
	Let $\lambda, \mu \in \Lambda$. Then we have:
	$$ \displaystyle
	\Ext_{\Otil}^{i}(\Mtil(\lambda), M(\mu)^{\vee}) =
					\begin{cases}
					\C
					& i=0,\lambda = \mu ;\\
					0 & \text{otherwise}.
					\end{cases}
	$$
	\end{Prop}
	\begin{proof}
		For $i=0$, we observe that
		$\Hom_{\Otil}(\Mtil(\lambda), M(\mu)^{\vee})
		\cong \Hom_{\OO}(M(\lambda), M(\mu)^{\vee})$
		and then the assertion 
		is well-known (see \cite{Hum} Theorem 3.3).

		For $i=1$, we consider an extension in $\Otil$:
		\begin{equation} \label{ext1}
		0 \to M(\mu)^{\vee} \to E \to \Mtil(\lambda) \to 0.
		\end{equation}
		If $\lambda \not < \mu$, then
		the weight $\lambda$ is a maximal weight of $E$.
		By the universal property of the deformed Verma module
		$\Mtil(\lambda)$, we see that the sequence (\ref{ext1})
		must be split.
		If $\lambda < \mu$, we apply the functor $D$
		defined above to the sequence (\ref{ext1}) to get:
		\begin{equation} \label{ext2}
		0 \to D(\Mtil(\lambda))
		\to D(E) \to M(\mu) \to
		\mathbf{R}^{1}D(\Mtil(\lambda))=0.
		\end{equation}
		Since $\lambda < \mu$,
		the weight $\mu$ is maximal in $D(E)$.
		Moreover, $R$ acts on $D(E)_{(\mu)}$ trivially
		because it is $1$-dimensional.
		Therefore by the universal property of the Verma
		module $M(\mu)$, we see that the sequence
		(\ref{ext2}) is split.
		Since $D(\Mtil(\lambda))^{\vee} \cong
		\Mtil(\lambda)$,
		by applying $(-)^{\vee}$ on the sequence
		(\ref{ext2}),
		we observe that the sequence (\ref{ext1})
		is also split.
		Thus we have 
		$\Ext_{\Otil}^{1}(\Mtil(\lambda), M(\mu)^{\vee}) = 0$
		for any $\lambda, \mu \in \Lambda.$

		The cases $i >1$ follow from the case $i=1$
		by a standard argument (cf.\  \cite{Hum} Section 6.12).
	\end{proof}
	
	We summarize the above results as follows. 
	
	\begin{Thm} \label{Thm:Otilaffin}
	 For each $\lambda \in \Lambda^{+},$ the block
	 $\Otil_{\lambda}$ of the category $\Otil$
	 is a complete affine highest weight category
	 with index poset $\Pi_{\lambda}$
	 satisfying the property $(\hat{\spadesuit})$.
	 The standard
	 (resp.\ proper standard, proper costandard)
	 modules of $\Otil_{\lambda}$
	 are the deformed Verma
	 modules $\Mtil(\mu)$
	 (resp.\ Verma modules $M(\mu)$, 
	 dual Verma modules $M(\mu)^{\vee}$).
	  \end{Thm}

\subsection{The Arakawa-Suzuki functor}
\label{AS}

	Let $V := \C^m$ be the vector representation
	of $\g = \gl_m$.
	We denote by $\{v_{i}\}$ the standard basis of $V$.
	The $\g$-action on $V$ is written explicitly as
		$ e_{ij} v_{k} = \delta_{jk} v_i .$
	We also write the dual space of $V$ by 
	$\bar{V} := \Hom_{\C}(V ,\C)$.

	Define the Casimir operator $\Omega \in 
	U(\g)\otimes U(\g)$ by 
	$\Omega := \sum_{1 \leq i,j \leq m}e_{ij}\otimes e_{ji}$.
	If $M_1$ and $M_2$ are both 
	$\g$-modules, 
	then $\Omega$ acts on 
	the tensor module $M_1 \otimes M_2$.
	Let $C := \sum_{1 \leq i, j \leq m} e_{ij} e_{ji}$ 
	be the Casimir element of $\g$, 
	which is a central element of $U(\g)$.
	We have $\Omega = \frac{1}{2}(\delta(C)-C \otimes 1 
	-1 \otimes C) \in U(\g) \otimes U(\g)$, 
	where
	$\delta : U(\g)\to U(\g)\otimes U(\g)$ denotes 
	the comultiplication of $U(\g)$ defined 
	by $\delta(X) := X\otimes 1 + 1 \otimes X$ 
	for $X \in \g$.
	Thus we have $\Omega \in \End_{\g}(M_1 \otimes M_2).$
	Note that the action of $\Omega$ 
	on the tensor representation
	$V \otimes V$ is the permutation of the tensor factors.

	Let $M_0, M_1,  \ldots , M_n$ be $\g$-modules.
	We define the linear operator $\Omega^{(a,b)}$ 
	for $0 \leq a , b \leq n$ on the tensor product
	$M_0 \otimes M_1 \otimes \cdots \otimes M_n$ 
	as $\Omega^{(a,b)} := \sum_{1 \leq i , j \leq m}
	e_{ij}^{(a)}e_{ji}^{(b)}$,
	where we define $X^{(a)} := 
	1^{\otimes a} \otimes X \otimes 1^{\otimes (n-a)} 
	\in U(\g)^{\otimes (n+1)}$ for $X \in \g$.
	Note that $\Omega^{(a,b)} = \Omega^{(b,a)}$ 
	and $\Omega^{(a,b)}$ commutes with the $\g$-action on
	$M_0 \otimes M_1 \otimes \cdots \otimes M_n$.

	\begin{Prop}[Arakawa-Suzuki \cite{AS}, \cite{Suzuki}] 
	\label{Prop:AS}
		Let $M$ be a $\g$-module.
		Then the formula
		$s_{i} = \Omega^{(i,i+1)} \,\, (1 \leq i < n), \,
		x_i =  \sum_{0\leq j < i}\Omega^{(i,j)} \,\, 
		(1 \leq i \leq n)
		$
		defines a right action of $H_n$ on 
		the tensor $\g$-module $M \otimes V^{\otimes n}$,
		and thus we have a $\C$-algebra homomorphism 
		$H_{n} \to 
		\End_{\g}(M \otimes V^{\otimes n})^{\mathrm{op}}$.
	\end{Prop}

	\begin{Def} \label{Def:AS}
	By taking $M=\Mtil(\rho)$ in Proposition \ref{Prop:AS} above,
	we define the Arakawa-Suzuki functor
	$$ \Phi =\Phi^{m}_{n} : \Otil(\gl_{m}) \to H_{n}\text{-Mod} 
	; \, M \mapsto
	\Hom_{\Otil}(\Mtil(\rho)\otimes V^{\otimes n},M).$$
	\end{Def}

	Note that the module $\Mtil(\rho)$ and
	hence $\Mtil(\rho)\otimes V^{\otimes n}$
	is projective in $\Otil$,
	since $\rho$ is a dominant weight.
 	Therefore the Arakawa-Suzuki functor
	$\Phi$ is an exact functor.
	
	\begin{Rem}
	The original definition \cite{AS}, \cite{Suzuki} of 
	Arakawa-Suzuki functor uses the 
	non-deformed Verma modules of
	general dominant weights,
	although we use 
	the deformed Verma module of
	a special dominant weight
	$\rho$ for simplicity.
	Our discussion below is valid
	for any regular dominant weights
	instead of $\rho$
	by some obvious modifications.
	\end{Rem}

	\begin{Lem} \label{Lem:commute}
		Let $M$ and $N$ be $\g$-modules.
		The Casimir operator defines the following linear operators
		$$\begin{array}{lll}
		\Omega^{*}_{M,V} : &\Hom_{\g}(M\otimes V, N) 
		\to \Hom_{\g}(M \otimes V, N) ;& f \mapsto 
		f \circ \Omega, \\
		\Omega_{\bar{V},N} :
		&\Hom_{\g}(M, \bar{V} \otimes N)
		\to \Hom_{\g}(M, \bar{V} \otimes N); 
		& f \mapsto \Omega \circ f .\\
		\end{array} $$
		We also have the natural isomorphism
		$\Hom_{\g}(M\otimes V, N) 
		\cong \Hom_{\g}(M, \bar{V} \otimes N) .$
		Then the following diagram is commutative.
		$$
		\xy
			\xymatrix{
				\Hom_{\g}(M\otimes V, N) \ar[r]^{\cong} 
				\ar[d]_{\Omega^{*}_{M,V}} & \Hom_{\g}(M, 
				\bar{V} \otimes N) 
				\ar[d]^{-\Omega_{\bar{V},N}-m}\\
				\Hom_{\g}(M\otimes V, N) \ar[r]^{\cong} & 
				\Hom_{\g}(M, \bar{V} \otimes N) }
		\endxy$$
	\end{Lem}
	\begin{proof}
	Recall $\Omega = 
	\frac{1}{2}(\delta(C)-C \otimes 1 -1 \otimes C) 
	\in U(\g) \otimes U(\g)$.
	For $f \in \Hom_{\g}(M\otimes V, N)$ and 
	$g \in \Hom_{\g}(M, \bar{V} \otimes N)$, we have
	\begin{align*}
		2\Omega^{*}_{M,V}(f) 
		&=2 f \circ \Omega \\
		&= f \circ 
		(\delta(C) - C_{M}\otimes 1 - 1 \otimes C_{V}) \\
		&= C_{N} \circ f - f \circ (C_{M}\otimes 1) - m f , \\
		2\Omega_{\bar{V},N}(g) &= 2\Omega \circ g \\
		&=   (\delta(C) - C_{\bar{V}}\otimes 1 - 1 \otimes C_{N})
		 \circ g \\
		&= g \circ C_{M} - m g - (1 \otimes C_{N}) \circ g,
	\end{align*}
	where $C_M$ denotes the action of the Casimir element 
	$C$ on $M$
	and we use $C_{V} = m$ and $C_{\bar{V}} = m$.
	From this, we obtain the assertion.
	\end{proof}

	\begin{Lem} \label{Lem:antiAS}
		Let $M ,N $ be $\g$-modules.
		\begin{enumerate}
		\item \label{Lem:antiAS:dual}
		The formula
		$ s_{i} = \Omega^{(i-1,i)}\,\, (1\leq i < n ), \,
		x_i =  -\sum_{i \leq j \leq n}\Omega^{(i-1,j)} -m \,\,
		(1 \leq i \leq n)
		$
		defines a left action of $H_n$ on the tensor product 
		$\g$-module $\bar{V}^{\otimes n} \otimes N$.
		\item \label{Lem:antiAS:compati}	
		The natural isomorphism
		$\Hom_{\g}(M \otimes V^{\otimes n}, N) 
		\cong \Hom_{\g}(M , \bar{V}^{\otimes n}\otimes N)$
		commutes with $H_n$-actions, where 
		we regard $\Hom_{\g}(M \otimes V^{\otimes n}, N) $
		as an $H_n$-module by Proposition \ref{Prop:AS} and
		$\Hom_{\g}(M , \bar{V}^{\otimes n}\otimes N)$ 
		as an $H_n$-module by (1).
		\end{enumerate}
	\end{Lem}
	\begin{proof}
		(\ref{Lem:antiAS:dual}) is similar to 
		Proposition \ref{Prop:AS}.
		(\ref{Lem:antiAS:compati}) 
		follows from Lemma \ref{Lem:commute}.
	\end{proof}

	We fix a dominant weight $\lambda \in \Lambda^{+}$
	satisfying $\lambda_{m} >0$ and denote its $\SG_{m}$-orbit
	by
	$\Pi_{\lambda} $ as before.
	We define an element $\beta \in Q^{+}$ by
	$
	 \beta := \sum_{i=1}^{m} \alpha(-i+1, \lambda_{i})
	$
	and we set $n := \mathtt{ht}(\beta). $
	For each $\mu \in \Pi_{\lambda}$,
	we associate a Kostant partition
	$\pi_{\mu} \in \KP(\beta)$
	by:
	$$
	\pi_{\mu} 
	:= \{\alpha(0, \mu_{1}), \alpha(-1, \mu_{2}), \ldots ,
	\alpha(-m+1, \mu_{m})\}.
	$$
	 \begin{Prop} \label{Prop:ASstd}
	 In the above notation,
	 for each $\mu \in \Pi_{\lambda}$, we have:
	 \begin{alignat*}{2}
	 	\Phi(\Mtil(\mu)) &\cong \std_{H}(\pi_{\mu}), \qquad &
	 	\Phi(M(\mu)) &\cong \pstd_{H}(\pi_{\mu}).
	 \end{alignat*}
	  \end{Prop}
	  \begin{proof}
	  For each module $M \in \Otil$, we compute as:
	  \begin{align*}
		\Hom_{\Otil}\left(\Mtil(\rho)\otimes 
		V^{\otimes n} , M\right)
		& \cong  \Hom_{\Otil}\left(\Mtil(\rho),  
		\bar{V}^{\otimes n} \otimes M\right) \\
		& \cong  \Hom_{\Otil}\left(\Mtil(\rho),  
		\mathop{\mathrm{pr}}\nolimits_{\rho}
		(\bar{V}^{\otimes n} \otimes M)\right) \\
		& \cong   \left( \mathop{\mathrm{pr}}\nolimits_{\rho}
		(\bar{V}^{\otimes n} \otimes M) \right)_{(0)}^{\n_{+}},
	\end{align*}
	where $\mathop{\mathrm{pr}}\nolimits_{\rho} : \Otil \to 
	\Otil_{\rho}$
	is the projection with respect to the block decomposition
	(Proposition \ref{Prop:Otildecomp}).
	Since the weight $0$ is the largest weight in the block
	$\Otil_{\rho}$, there are natural isomorphisms:
	$$
	\left( \mathop{\mathrm{pr}}\nolimits_{\rho}
		(\bar{V}^{\otimes n} \otimes M) \right)_{(0)}^{\n_{+}}
		\cong
		\left( \mathop{\mathrm{pr}}\nolimits_{\rho}
		(\bar{V}^{\otimes n} \otimes M) \right)_{(0)}
		\cong
		\left(
		(\bar{V}^{\otimes n} \otimes M)/\n_{-} \right)_{(0)}.
	$$
	Set $M=\Mtil(\mu)$, then we have
	\begin{align*}
	\Phi(\Mtil(\mu)) &\cong
	\left(
		(\bar{V}^{\otimes n} \otimes \Mtil(\mu))/\n_{-} 
		\right)_{(0)} \\
	&\cong
	\left(
		(U(\g)\otimes_{U(\be)}
		(\bar{V}^{\otimes n} \otimes R_{\mu - \rho}))/\n_{-} 
		\right)_{(0)} \\
	& \cong
		(\bar{V}^{\otimes n} \otimes R_{\mu - \rho})_{(0)} 
	\end{align*}
	where we use the tensor identity for
	the second equality.
	Let $u$ be the vector belonging to $(\bar{V}^{\otimes n} 
	\otimes R_{\mu - \rho})_{(0)} $
	defined by
	$u := \e_{1}^{\otimes n_{1}} \otimes \e_{2}^{\otimes n_{2}}
	\otimes \cdots \otimes \e_{m}^{\otimes n_{m}} \otimes 1,$
	where
	we put
	$n_{i} := \mu_{i} + i -1 = \mu_{i} - \rho_{i}$ and
	$\{\e_{i}\}$ is the basis of $\bar{V}$
	dual to the standard basis $\{v_{i}\} $ of $V$.
	Note that we have
	$e_{ij}\e_{k}=-\delta_{ik}\e_{j}$ for $1 \leq i,j,k \leq m.$

	For each $1 \leq k \leq m,$
	we define the number $a_{k}$ by
	$a_{k} := \sum_{l=1}^{k-1} n_{l} +1$
	as in Remark \ref{Rem:stdH}.
	We compute the action of $x_{a_{k}} \in H_{n}$
	on the vector $u$ as:
	$$
	\begin{array}{cl} 
	&\displaystyle(-x_{a_{k}}-m) u  \\
	\stackrel{\ref{Lem:antiAS}(\ref{Lem:antiAS:compati})}{=} 
	& \displaystyle \sum_{1 \leq i,j \leq m}
	\e_{1}^{\otimes n_{1}}\otimes \cdots
	\otimes \e_{k-1}^{\otimes n_{k-1}}
	\otimes e_{ij} \e_{k} \otimes e_{ji}
	(\e_{k}^{\otimes (n_{k}-1)}\otimes \cdots \otimes 
	\e_{m}^{\otimes n_{m}} \otimes 1)\\
	=& - 
	\displaystyle \sum_{j=1}^{m}
	\e_{1}^{\otimes n_{1}}\otimes \cdots
	\otimes \e_{k-1}^{\otimes n_{k-1}}
	\otimes \e_{j} \otimes e_{jk}
	(\e_{k}^{\otimes (n_{k}-1)}\otimes \cdots \otimes 
	\e_{m}^{\otimes n_{m}} \otimes 1)\\
	\stackrel{\bmod{\n_{-}}}{\equiv}&
	\displaystyle \sum_{j=k+1}^{m}
	e_{jk}(
	\e_{1}^{\otimes n_{1}}\otimes \cdots
	\otimes \e_{k-1}^{\otimes n_{k-1}}
	\otimes \e_{j} )\otimes
	\e_{k}^{\otimes (n_{k}-1)}\otimes \cdots \otimes 
	\e_{m}^{\otimes n_{m}} \otimes 1\\
	& -
	\displaystyle \sum_{j=1}^{k}
	\e_{1}^{\otimes n_{1}}\otimes \cdots
	\otimes \e_{k-1}^{\otimes n_{k-1}}
	\otimes \e_{j} \otimes e_{jk}
	(\e_{k}^{\otimes (n_{k}-1)}\otimes \cdots \otimes 
	\e_{m}^{\otimes n_{m}} \otimes 1) 
	\\
	=&
	\displaystyle \sum_{j=k+1}^{m}
	\e_{1}^{\otimes n_{1}}\otimes \cdots
	\otimes \e_{k-1}^{\otimes n_{k-1}}
	\otimes e_{jk}\e_{j} \otimes
	\e_{k}^{\otimes (n_{k}-1)}\otimes \cdots \otimes 
	\e_{m}^{\otimes n_{m}} \otimes 1\\
	&-
	\e_{1}^{\otimes n_{1}}\otimes \cdots
	\otimes \e_{k-1}^{\otimes n_{k-1}}
	\otimes \e_{k} \otimes e_{kk}
	(\e_{k}^{\otimes (n_{k}-1)}\otimes \cdots \otimes 
	\e_{m}^{\otimes n_{m}} \otimes 1) \\
	=&
	(k-m)u
	+(n_{k}-1 - n_{k}-z_{k})u \\
	=& 
	(k-m-1-z_{k})u.
	\end{array}
	$$
	Thus we have $x_{a_{k}} u = (-k+1 + z_{k})u.$
	Since the parabolic subgroup $\SG_{n_{1}} \times 
	\cdots \times \SG_{n_{m}}$
	fixes the vector $u$,
	the element
	$x_{l}$ acts on by the multiplication of the scalar
	$-k+1 + (l-a_{k}) + z_{k}$
	if $a_{k} \leq l < a_{k+1}$.
	This shows that
	there is an $H_{n}$-homomorphism
	$\std_{H}(\pi_{\mu})
	\to
	(\bar{V}^{\otimes n} )_{(-\mu + \rho)}
	\otimes R_{\mu - \rho}
	$
	which sends the generator $1$ to the vector $u$
	(see Remark \ref{Rem:stdH}).
	By comparing the rank of $R$-modules,
	we see that this is actually an isomorphism.
	Thus we have $\Phi(\Mtil(\mu)) \cong \std_{H}(\pi_{\mu})$.
	The assertion $\Phi(M(\mu)) \cong 
	\pstd_{H}(\pi_{\mu})$ is obtained from this
	by taking $\m$-coinvariants.
	We remark that this latter assertion
	is a special case of
	the results of Arakawa-Suzuki \cite{AS}, \cite{Suzuki}.
	  \end{proof}

	  \begin{Lem} \label{Lem:ASdual}
	  For each $M \in \OO$, we have
	  $\Phi(M^{\vee}) \cong \Phi(M)^{\circledast}.$
	  In particular, we have
	    $\Phi(M(\mu)^{\vee}) \cong \pcstd_{H}(\pi_{\mu})$
	    for each $\mu \in \Pi_{\lambda}.$
	  \end{Lem}
	  \begin{proof}
	  We remark that $N^{\vee} \cong N$ holds for
	  any finite dimensional $\g$-module $N$
	  and $(N\otimes M )^{\vee} \cong N \otimes M^{\vee}$
	  for any $M \in \OO$.
	  We compute as:
	  \begin{align*}
	  \Phi(M^{\vee}) & \cong \Hom_{\Otil}(\Mtil(\rho)\otimes
	  V^{\otimes n}, M^{\vee}) \\
	  & \cong \Hom_{\OO}(M(\rho)\otimes V^{\otimes n}, 
	  M^{\vee}) \\
	  & \cong
	  \Hom_{\OO}(M(\rho), \bar{V}^{\otimes n}\otimes 
	  M^{\vee}) \\
	   & \cong
	  \Hom_{\OO}(M(\rho), (\bar{V}^{\otimes n}\otimes 
	  M)^{\vee}) \\
	  & \cong
	 \left( (\bar{V}^{\otimes n}\otimes M)/\n_{-} 
	 \right)_{(0)}^{*}\\
	 & \cong \Phi(M)^{\circledast},
	  \end{align*}
	  which proves the assertion.
	  \end{proof}

	Finally, we prove the main theorem of this section.

	\begin{Thm} \label{Thm:main}
	The Arakawa-Suzuki functor induces 
	a fully faithful functor
	$\Phi : \Otil_{\lambda} \to \hat{H}_{\beta} \mof$
	between complete affine highest weight categories.
	\end{Thm}
	\begin{proof}
	Note that the category $\Otil_{\lambda}$
	is the Serre subcategory 
	of $\Otil$ 
	generated by 
	deformed Verma modules $\Mtil(\mu)$
	for $\mu \in \Pi_{\lambda}$.
	By Proposition \ref{Prop:ASstd}, we see
	for each $M \in \Otil_{\lambda}$ the $H_{n}$-action
	on $\Phi(M)$ is naturally extended to the action of
	$\hat{H}_{\beta}.$
	Therefore we can regard $\Phi$ as the exact functor
	$\Phi : \Otil_{\lambda} \to \hat{H}_{\beta}\mof.$
	
	We have to show that this functor
	$\Phi$ is fully faithful.
	Let us consider the subset 
	$\KP(\beta)_{m} \subset \KP(\beta)$
	which consists of Kostant partitions of 
	$\beta$ consisting of $m$ positive roots
	i.e.\ $\KP(\beta)_{m} := 
	\{\pi \in \KP(\beta) \mid |\pi| = m\}.$
	It is easy to show
	that the correspondence
	$\Pi_{\lambda} \to \KP(\beta)_{m}; \, \mu \mapsto \
	\pi_{\mu}$
	is an isomorphism of posets and
	$\KP(\beta)_{m} = \{ \pi \in \KP(\beta) \mid \pi \unlhd 
	\pi_{\lambda} \}$,
	which is saturated.
	Then the full subcategory $\hat{H}_{\beta}\mof_{m}
	\subset \hat{H}_{\beta}\mof$ consisting of
	modules $M$ with $\supp(M) \subset \KP(\beta)_{m}$
	is an affine highest weight category
	whose standard 
	(resp.\  proper standard, proper costandard) modules
	are $\std_{H}(\pi_{\mu})$ (resp.\ 
	$\pstd_{H}(\pi_{\mu}), \pcstd_{H}(\pi_{\mu})$).
	Thanks to Proposition \ref{Prop:ASstd} and 
	Lemma \ref{Lem:ASdual},
	we can apply
	the complete version of 
	Theorem \ref{Thm:equiv} to the functor
	$\Phi: \Otil_{\lambda} \to \hat{H}_{\beta}\mof_{m}$
	to complete the proof.
	\end{proof}
	
	\begin{Rem}
	The subcategory
	 $\hat{H}_{\beta}\mof_{m} 
	 \subset \hat{H}_{\beta}\mof$ 
	 is naturally equivalent to the module category
	 over the maximal quotient algebra of $\hat{H}_{\beta}$
	 whose support as a left $\hat{H}_{\beta}$-module
	 is contained in $\KP(\beta)_{m}$
	 (cf.\ \cite{affine} Lemma 3.13).
	 Therefore the above proof of Theorem 
	 \ref{Thm:main} says that the block 
	 $\Otil_{\lambda}$ 
	 is equivalent to the module category 
	 over a quotient algebra of $\hat{H}_{\beta}.$
	 \end{Rem}


\end{document}